\documentclass[onefignum,onetabnum]{siamonline190516}

\usepackage{amsmath,amsfonts,graphics,amssymb,bm}
\usepackage{diagbox}
\usepackage{mathtools}
\newcounter{rmrk}[section]

\numberwithin{equation}{section}

\newtheorem{assumption}{Assumption}

\numberwithin{equation}{section}
\newcommand{\norm}[1]{\left\Vert#1\right\Vert}
\newcommand{\abs}[1]{\left\vert#1\right\vert}
\newcommand{\set}[1]{\left\{#1\right\}}

\newcommand{\pspace}{\mathbf{\Lambda}}
\newcommand{\dspace}{\mathbf{\mathcal{D}}}
\newcommand{\pmeas}{\mu_{\pspace}}
\newcommand{\dmeas}{\mu_{\dspace}}
\newcommand{\pborel}{\mathcal{B}_{\pspace}}
\newcommand{\dborel}{\mathcal{B}_{\dspace}}

\newcommand{\pspacedim}{k}
\newcommand{\dspacedim}{m}

\newcommand{\cbprior}{initial}
\newcommand{\cbpriorabrv}{i} 
\newcommand{\cbposterior}{updated}
\newcommand{\cbposteriorabrv}{u} 

\newcommand{\priormeas}{P_{\pspace}^{{\cbpriorabrv}}}
\newcommand{\postmeas}{P_{\pspace}^{{\cbposteriorabrv}}}
\newcommand{\priordens}{\pi_{\pspace}^{{\cbpriorabrv}}}
\newcommand{\postdens}{\pi_{\pspace}^{{\cbposteriorabrv}}}

\newcommand{\pfpriordens}{\pi_{\dspace}^{Q}}

\newcommand{\obsmeas}{P_{\dspace}}
\newcommand{\obsdens}{\pi_{\dspace}}

\newcommand{\postdensa}{{\pi}_{\pspace}^{u,n}}

\newcommand{\pfpriordensa}{{\pi}_{\dspace}^{{Q_n}}}

\def\qoi{Q(\lambda)}
\def\qoia{{Q}_n(\lambda)}
\def\ra{r_n(\lambda)}

\title{Convergence of Probability Densities using Approximate Models for Forward and Inverse Problems in Uncertainty Quantification: Extensions to $L^p$}
\date{\today}
\author{T.~Butler\thanks{University of Colorado Denver, Department of Mathematical and Statistical Sciences} \and T.~Wildey\thanks{Sandia National Laboratories, Center for Computing Research ({\tt tmwilde@sandia.gov}). The views expressed in the article do not necessarily represent the views of the U.S. Department of Energy or the United States Government.  Sandia National Laboratories is a multimission laboratory managed and operated by National Technology and Engineering Solutions of Sandia, LLC., a wholly owned subsidiary of Honeywell International, Inc., for the U.S. Department of Energy's National Nuclear Security Administration under contract DE-NA-0003525.} \and W.~Zhang\thanks{University of Colorado Denver, Department of Mathematical and Statistical Sciences} }

\begin{document}
\maketitle

\newcommand{\slugmaster}{%
\slugger{sisc}{xxxx}{xx}{x}{x--x}}

\begin{abstract}
A previous study analyzed the convergence of probability densities for forward and inverse problems when a sequence of approximate maps between model inputs and outputs converges in $L^\infty$.
This work generalizes the analysis to cases where the approximate maps converge in $L^p$ for any $1\leq p < \infty$.
Specifically, under the assumption that the approximate maps converge in $L^p$, the convergence of probability density functions solving either forward or inverse problems is proven in $L^q$ where the value of $1\leq q<\infty$ may even be greater than $p$ in certain cases.
This greatly expands the applicability of the previous results to commonly used methods for approximating models (such as polynomial chaos expansions) that only guarantee $L^p$ convergence for some $1\leq p<\infty$.
Several numerical examples are also included along with numerical diagnostics of solutions and verification of assumptions made in the analysis.
\end{abstract}

\begin{keywords}
inverse problems, uncertainty quantification, density estimation, $L^p$ convergence, surrogate modeling, approximate modeling
\end{keywords}

\begin{AMS}
60F25, 60H30, 60H35
\end{AMS}

\pagestyle{myheadings} \thispagestyle{plain} \markboth{T.~Butler, T.~Wildey, W.~Zhang}{$L^p$ convergence of push-forward/pull-back densities}

\section{Introduction}\label{sec:intro}

Many uncertainty quantification (UQ) problems involve the propagation of uncertainties, described using probability measures, between the input and output spaces of a computational model.
When the input (or output) space is a measure space with a specified {dominating} measure used to describe the sizes (not probabilities) of sets, we can use the Radon-Nikodym derivative to uniquely express any probability measure that is absolutely continuous with respect to the dominating measure as a probability density function (PDF or density).
We are interested in densities solving two types of UQ problems in this work.
The first density is on the output space of a model defined by the push-forward of an initial probability density specified on the inputs of a model.
This represents a solution to a forward UQ problem where the goal is to predict probable outputs of a model using some initial/prior knowledge of model inputs.
The second density is on the input space of a model defined by the pullback of an observed (or specified) probability density on model outputs.
This represents a solution to a particular type of inverse UQ problem where the goal is to determine a density on model inputs whose push-forward through the model reconstructs the observed (or specified) probability density.
In a sense, this represents a type of probabilistic calibration of inputs.

While the existence, uniqueness, and stability of push-forward and pullback densities are studied in~\cite{BJW18a}, such densities are often numerically estimated using a finite number of computational solutions of the model relating input and output spaces.
We are particularly interested in cases where the outputs are defined by functionals applied to the solution to the model, which defines a (measurable) map between input and output spaces.
Thus, we focus on the underlying, and often implicitly defined, map between input and output spaces instead of the actual model itself.

In \cite{BJW18b}, we show that if a sequence of approximate maps from input to output spaces converges in $L^\infty$, then the associated sequences of either push-forward or pullback densities converge.
However, we cannot always guarantee, nor is it always reasonable to assume, that a sequence of approximate maps converges in $L^\infty$.
For instance, we may construct approximate maps using generalized polynomial chaos expansions (PCEs)~\cite{GhanemSpanos,XK-02}, sparse grid interpolation~\cite{barthelmann00,Ma:2009:AHS:1514432.1514547} and Gaussian process models~\cite{scheuerer_schaback_schlather_2013,rasmussen2006}.
Take PCEs as an example, according to the Cameron-Martin Theorem, the Hermite-chaos expansion converges for any arbitrary random process with finite second-order moments, i.e., the convergence of approximate maps occurs in $L^2$.

The primary contribution of this work is the generalization of convergence results for push-forward and pullback densities defined by a sequence of approximate maps that converge in $L^p$ with $1\leq p < \infty$.
The generalization of results is by no means trivial and includes both subtle, and in some places dramatic, changes to the analysis presented in \cite{BJW18b}.

We briefly recall some familiar results from measure theory (see, e.g., \cite{Folland} for a concise introduction or \cite{Bogachev_vol1} for a more thorough treatise) that provide some insight into how the convergence theory developed for densities associated with maps that converge in $L^\infty$ can be extended to a subsequence of densities associated with maps that converge in $L^p$.
First, it is well known that if a sequence of measurable functions converges in $L^p$, then there exists a subsequence which converges {\em almost everywhere} (a.e.).
In other words, the subsequence converges except on a set of measure zero.
Applying Egorov's Theorem, this subsequence converges {\em almost} uniformly (i.e., in $L^\infty$).
The measure-theoretic distinction between a property holding a.e.~as opposed to holding in an almost sense is subtle when first encountered.
A property is said to be true in an almost sense if for every $\epsilon>0$ there exists a set of measure less than $\epsilon$ such that the property holds on the complement of this ``small'' set.
This implies that the results of \cite{BJW18b} may apply to a subsequence of the $L^p$ maps restricted to a ``large'' subset of the spaces.
However, while the existence results of the subsequence and subset are powerful theoretical tools, they are of little use practically. Specifically, it is not at all clear, in general, how one is to determine either the subsequence of maps or the subset of the spaces on which to carry out the analysis.
Subsequently, it is unclear which part of the sequence of approximate maps or subsets of the spaces we can even apply the results from \cite{BJW18b} when the maps converge in $L^p$.

This paper provides convergence results directly on the entire sequence (not a subsequence) of approximate solutions to both the forward and inverse problems when using approximate maps that converge in $L^p$.
We prove several convergence results for push-forward densities including the convergence of approximate push-forward densities evaluated at data points predicted by the approximate models.
The tightness of push-forward probability measures associated with these densities is also proven which permits the proof of convergence of approximate updated densities.
Most results are framed as convergence in $L^p$ both for simplicity and to coincide with the convergence of the approximate maps.
However, it is worth noting that either finite measurability of the parameter space or convergence of the approximate maps in $L^s$ for any $1\leq s\leq p$ immediately results in $L^s$ convergence of both push-forward densities and updated densities for any $1\leq s\leq p$ by standard results in measure theory.

The remainder of the paper is organized as follows.
In Section~\ref{sec:notation}, we summarize most of the notation and terminology used in this work.
The forward and inverse problem analyses follow in Sections~\ref{sec:forward-theory} and \ref{sec:inverse-theory}, respectively.
We discuss the impact of computational estimates of densities in Section~\ref{sec:other-errors}.
Section~\ref{sec:verifying-assumptions} discusses how we numerically verify the two main assumptions of this work when they hold in a stronger sense. 
The main numerical results that employ PCEs follow in Section~\ref{sec:applications}.
Concluding remarks and acknowledgments are given in Sections~\ref{sec:conclusions} and~\ref{sec:acknowledge}, respectively. 
Appendix~\ref{app:a} includes some additional discussion on the generality of the assumptions used in this work in the context of an example involving a singular push-forward density.
This includes discussion on how we may utilize the numerical tools and diagnostics described in Section~\ref{sec:verifying-assumptions} to identify sources of error and improve accuracy of estimated densities.

\section{The spaces and maps}\label{sec:notation}

Here, we summarize some common terminology, notation, and implicit assumptions used for both the forward and inverse analysis in this work.
Denote by $\pspace\subset\mathbb{R}^{\pspacedim}$ the space of all {\em physically possible} inputs to the model, which we refer to as parameters.
A quantity (or quantities) of interest (QoI) refers to the functional(s) applied to the solution space of the model corresponding to scalar (or vector-valued) outputs.
Depending on the problem, a QoI map may correspond to either output values we wish to predict or values for which we have already obtained data, e.g., for the purpose of model validation or calibration.
The parameter-to-QoI map (often referred to simply as the QoI map) is denoted by $Q(\lambda):\pspace\to\dspace\subset\mathbb{R}^\dspacedim$ to make explicit the dependence on model parameters.
Let $(Q_n)$ denote a sequence of approximate QoI maps, $Q_n:\pspace\to\dspace$.
Here, $\dspace$ denotes the range of all QoI maps indicating the space of {\em physically possible} output data that the model, or any of its approximations, can predict.

Assume that $(\pspace, \pborel, \pmeas)$ and $(\dspace,\dborel,\dmeas)$ are both measure spaces.
Here, $\pborel$ and $\dborel$ denote the Borel $\sigma$-algebras inherited from the metric topologies on their respective spaces, and $\pmeas$ and $\dmeas$ denote the dominating measures for which probability densities (i.e., Radon-Nikodym derivatives of probability measures) are defined on each space.
It is implicitly assumed that every QoI map is a measurable map between these measure spaces.
The assumption, stated explicitly throughout this work, is that $Q_n\to Q$ in $L^p(\pspace)$.
The convergence of densities is then considered on either $L^r(\dspace)$ or $L^p(\pspace)$.
By $L^r(\dspace)$ or $L^p(\pspace)$, unless stated otherwise, it is implicitly assumed that the integrals are with respect to the dominating measures $\dmeas$ or $\pmeas$, respectively.
A practical assumption implicitly made by computations involving finite sampling of these maps and standard density approximation techniques is that each map is also piecewise smooth.

\section{Forward problem analysis}\label{sec:forward-theory}

In this section we analyze the convergence of push-forward densities obtained from solving a forward UQ problem using a sequence of approximate maps which converge to the exact map in the $L^p$ sense where $p$ is fixed and $1\leq p <\infty$.

\subsection{Problem definition}

\begin{definition}[Forward Problem and Push-Forward Measure]\label{def:forward-problem}
  Given a probability measure $P_\pspace$ on $(\pspace,\pborel)$ that is absolutely continuous with respect to $\pmeas$ and admits a density $\pi_\pspace$, the forward problem is the determination of the push-forward probability measure
\begin{equation*}
P^{Q}_\dspace(A) = P_\pspace(Q^{-1}(A)), \quad \forall A\in \dborel.
\end{equation*}
on $(\dspace,\dborel)$ that is absolutely continuous with respect to $\dmeas$ and admits a density $\pi_\dspace^{Q}$.
\end{definition}

We emphasize that the same probability measure $P_\pspace$ (which we later refer to as an {\em initial} measure for the inverse problem) is considered for any forward problem.
It is only the QoI maps that may change between problems.

\subsection{Convergence analysis}\label{sec:fp_approx_models}

Denote by $(\pfpriordensa)$ the sequence of approximate push-forward densities.
In \cite{BJW18b}, the major takeaway of the forward analysis is that $\pfpriordensa(\qoia)\to\pfpriordens(\qoi)$ in $L^\infty(\pspace)$ when $Q_n\to Q$ in $L^\infty(\pspace)$.
The interpretation is that the approximate push-forward densities evaluated at the approximate QoI converge {\em essentially uniformly} to the exact push-forward density evaluated at the exact QoI.
The proof in \cite{BJW18b} requires connecting several concepts, assumptions, and results, which we summarize below to highlight exactly where  the previous analysis fails to apply in this work.

The assumption of $Q_n\to Q$ in $L^\infty(\pspace)$ implies that the push-forward probability measures converge, i.e., there is convergence in distribution.
However, this is not sufficient to guarantee convergence of the densities.
Specifically, we cannot guarantee that $\pfpriordensa(q)\to \pfpriordens(q)$ converges even pointwise on $\dspace$ let alone the desired result that $\pfpriordensa(\qoia)\to\pfpriordens(\qoi)$ in $L^\infty(\pspace)$.
For this, we need an additional assumption about the type of sequence of approximate densities we are dealing with that removes pathological cases from consideration.
The work of \cite{Sweeting_86} defines sufficient conditions in terms of weaker asymptotic notions of equicontinuity and uniform equicontinuity.
\begin{definition}\label{def:a.e.c}
Using similar notation from \cite{Sweeting_86}, we say that a sequence of real-valued functions $(u_n)$ defined on $\mathbb{R}^\pspacedim$ is {\em asymptotically equicontinuous (a.e.c.)} at $x\in\mathbb{R}^\pspacedim$ if
\begin{equation*}
\forall \epsilon>0,\,  \exists \delta(x,\epsilon)>0, n(x,\epsilon) \text{ s.t. } \abs{y-x}<\delta(x,\epsilon), n>n(x,\epsilon) \Rightarrow \abs{u_n(y)-u_n(x)}<\epsilon.
\end{equation*}
If $\delta(x,\epsilon)=\delta(\epsilon)$ and $n(x,\epsilon)=n(\epsilon)$, then we say that the sequence is {\em asymptotically uniformly equicontinuous (a.u.e.c.)}.
\end{definition}

In \cite{BJW18b}, we assume that $(\pfpriordensa)$ are uniformly bounded and asymptotically uniformly equicontinuous.
Then, the uniform continuity of the exact push-forward can be obtained by a result from \cite{Sweeting_86}.
To obtain the desired result, the triangle inequality is applied to
\begin{equation*}
	\norm{\pfpriordensa(\qoia)-\pfpriordens(\qoi)}_{L^\infty(\pspace)}
\end{equation*}
resulting in an upper bound of the form
\begin{equation*}
	\norm{\pfpriordensa(\qoia)-\pfpriordens(\qoia)}_{L^\infty(\pspace)} + \norm{\pfpriordens(\qoia)-\pfpriordens(\qoi)}_{L^\infty(\pspace)}.
\end{equation*}
The first term is recongized as being identical to
\begin{equation*}
	\norm{\pfpriordensa(q)-\pfpriordens(q)}_{L^\infty(\dspace)},
\end{equation*}
and subsequently goes to zero by a result from \cite{Sweeting_86}.
The second term goes to zero by combining the uniform continuity of $\pfpriordens$ along with the $L^\infty(\pspace)$ convergence of the approximate maps.
If $Q_n\to Q$ in $L^p(\pspace)$ for $1\leq p<\infty$, then this conclusion is not in general true.
However, as we show below, under certain conditions it is true that the approximate push-forward densities evaluated at the approximate QoI converge to the exact push-forward density evaluated at the exact QoI in $L^s(\pspace)$ for $1\leq s\leq p$.
Below, we give our first formal assumption.

\begin{assumption}\label{assump:almost_aec}
The sequence of approximate push-forward densities, $(\pfpriordensa)$ is almost uniformly bounded and almost a.e.c.
\end{assumption}

In Assumption~\ref{assump:almost_aec}, and in the rest of this work, we make use of the term ``almost'' in the precise measure-theoretic sense.
Specifically, a property is said to hold in an almost sense on a measure space $(X,\mathcal{B}_X, \mu_X)$ with $\mu_X(X)<\infty$ if for any $\epsilon>0$, there exists $A_{\epsilon}\in \mathcal{B}_X$ such that $\mu_X( A_{\epsilon})<\epsilon$ and the property holds on $X\backslash A_{\epsilon}$.
This is loosely interpreted as stating that a property holds except on sets of small, but positive, measure.
This is a familiar concept in measure theory (e.g., see Egorov's or Lusin's theorem \cite{Folland,Bogachev_vol1}).
If $\mu_X(X)=\infty$, then we instead say that a property holds in the almost sense if for any $A\in\mathcal{B}_X$ with $\mu_X(A)<\infty$ and $\epsilon>0$, there exists $A_\epsilon\in\mathcal{B}_X$ such that $\mu_X(A_\epsilon)<\epsilon$ and the property holds on $A\backslash A_\epsilon$.

That the properties in Assumption~\ref{assump:almost_aec} hold in an almost sense is similar to the concept of {\em tightness} of measures.
For example, when $X$ has a topology (so that $\mathcal{B}_X$ represents the Borel $\sigma$-algebra on $X$), then a family of probability measures $\mathcal{P}$ on $(X,\mathcal{B}_X)$ is considered tight if for any $\epsilon>0$ there exists compact $K_\epsilon \in\mathcal{B}_X$ such that for any $\mathbb{P}\in\mathcal{P}$, $\mathbb{P}(K_\epsilon) > 1-\epsilon$.
In other words, there exists a compact set containing most of the probability no matter which probability measure is considered.
Assumptions of tightness are often assumed in classical results involving the convergence of measures.
For example, Prokhorov's theorem \cite{Bogachev_vol2} states that tightness of probability measures is a necessary and sufficient condition for the precompactness of these measures in the topology of weak convergence.

The next two lemmas describe the types of convergence of the approximate push-forward densities that occur on $\dspace$ under certain conditions.

\begin{lemma}\label{lem:pf_prior_almostauec_1}
Suppose $1\leq p<\infty$ and $\qoia\to \qoi$ in $L^p(\pspace)$.
If Assumption~\ref{assump:almost_aec} holds, then
\begin{equation}\label{limit:ptwise_pf_D}
	\pfpriordensa(q) \to \pfpriordens(q)  \indent \text{almost on } \dspace.
\end{equation}
Furthermore, 
for any compact subset $D_c\subset\dspace$ and $1\leq r\leq \infty$,
\begin{equation}\label{limit:q_pf_Dc_1}
	\pfpriordensa(q)\to \pfpriordens(q) \indent \text{almost in }  L^r(D_c).
\end{equation}
\end{lemma}

Before we prove this lemma, we recall two useful results from measure theory.
First, if $X$ is a Euclidean space and $A\in\mathcal{B}_X$ with $\mu_X(A)<\infty$, then, for any $\epsilon>0$, there exists an open $G\supset A$ and a compact $K\subset A$ such that $\mu_X(G\backslash K)<\epsilon$.
In other words, any set of finite measure can be approximated arbitrarily well (in measure) by either an open set containing it or a compact set contained within it.
Thus, without loss of generality, the set where the uniform bounded and a.e.c~criteria in Assumption~\ref{assump:almost_aec} do {\em not} hold can be chosen as an open set.
This is done in the proof below where this set is denoted by $N_\delta$. 
Second, the general form of Lusin's theorem applies to the measure spaces considered in this work.
In the context of this work, this implies that the densities are almost continuous functions with compact support.
Subsequently, densities defined on either $\pspace$ or $\dspace$ are almost in $L^r$ for $1\leq r\leq \infty$.

\begin{proof}
Since $\qoia\to \qoi$ in $L^p(\pspace)$, $\qoia$ converges weakly to $\qoi$.
This along with Assumption~\ref{assump:almost_aec} implies that $\pfpriordensa(q)$ almost converges to $\pfpriordens(q)$ using Theorem 1 from \cite{Sweeting_86} in an almost sense.
This proves \eqref{limit:ptwise_pf_D}.

Let $\delta>0$ and consider any compact subset $D_c\subset \dspace$. 
By Assumption~\ref{assump:almost_aec} and the fact that $\pfpriordens(q)$ is almost in $L^{\infty}(\dspace)$, there exists an open set $N_{\delta}$ such that $\dmeas(N_{\delta})<\delta$, $(\pfpriordensa)$ is uniformly bounded and a.e.c. on $\dspace\backslash N_{\delta}$, and $\pfpriordens(q)$ is in $L^\infty(\dspace\backslash N_\delta)$.
By the compactness of $D_c$ and openness of $N_\delta$, $(\pfpriordensa)$ is a.u.e.c. on $D_c\backslash N_{\delta}$.
Then, by Theorem 2 from \cite{Sweeting_86},
\begin{equation*}
\pfpriordensa(q) \to \pfpriordens(q) \text{ in }\ L^{\infty}(D_c\backslash N_{\delta}).
\end{equation*}
Finally, for any $1\leq r<\infty$, the embedding $L^{\infty}(D_c\backslash N_{\delta}) \subset L^r(D_c\backslash N_{\delta})$ implies $\pfpriordensa(q) \to \pfpriordens(q)$ in $L^r(D_c\backslash N_{\delta})$ which proves \eqref{limit:q_pf_Dc_1}.
\end{proof}

To extend the convergence in~\eqref{limit:q_pf_Dc_1} to all of $D_c$, we require some mechanism for controlling the size of the $L^r$-norms of the push-forward densities on the $N_\delta$ set.
In other words, we need to ensure that the probability mass of the push-forwards is not collecting in data sets of arbitrarily small measure.
An assumption of uniform integrability avoids such pathological families of densities.

\begin{lemma}\label{lem:pf_prior_almostauec}
Suppose $1\leq p<\infty$ and $\qoia\to \qoi$ in $L^p(\pspace)$.
If Assumption~\ref{assump:almost_aec} holds and the family of push-forward densities defined by these maps are uniformly integrable in $L^r(\dspace)$ for some $1\leq r<\infty$, then for any compact subset $D_c\subset\dspace$
\begin{equation}\label{limit:q_pf_Dc}
	\pfpriordensa(q)\to \pfpriordens(q) \indent \text{ in }  L^r(D_c).
\end{equation}
\end{lemma}

\begin{proof}
Let $\epsilon>0$.
Use the uniform integrability in $L^r(\dspace)$ to choose $\delta>0$ such that the integral of any push-forward density raised to the $r$ power over a set $A\in\dborel$ with $\dmeas(A)<\delta$ is bounded by $\epsilon^r/2$.
Use Assumption~\ref{assump:almost_aec} to choose the $N_\delta$ set.
Using the fact that
\begin{equation*}
	\norm{\pfpriordensa(q)-\pfpriordens(q)}_{L^r(D_c)} = \left[\norm{\pfpriordensa(q)-\pfpriordens(q)}_{L^r(D_c\backslash N_\delta)}^r + \norm{\pfpriordensa(q)-\pfpriordens(q)}_{L^r(N_\delta)}^r\right]^{1/r}
\end{equation*}
along with~\eqref{limit:q_pf_Dc_1} to bound the first term by $\epsilon^r/2$ for sufficiently large $n$ proves~\eqref{limit:q_pf_Dc}.
\end{proof}

The next lemma immediately follows by recalling the classical result that $C^\infty_c(\mathbb{R}^n)$ (i.e., the space of infinitely differentiable functions with compact support) is dense in $L^q(\mathbb{R}^n)$ \cite{Adams}.

\begin{lemma}\label{lem:lipschitz_approx}
Lipschitz continuous functions with compact support are dense in $L^q(X)$ for $X\in\set{\pspace,\dspace}$ and any $1\leq q<\infty$.
\end{lemma}

We now state the analog to the main result of the forward analysis given in \cite{BJW18b}.
The interpretation is that the approximate push-forward densities evaluated at the approximate QoI converge to the exact push-forward density evaluated at the exact QoI in $L^p(\pspace)$.

\begin{theorem}[$L^p$ Convergence of Push-Forward Densities]\label{thm:lam_pf_D}
Suppose $1\leq p<\infty$ and $\qoia\to \qoi$ in $L^p(\pspace)$.
If Assumption~\ref{assump:almost_aec} holds  and the family of push-forward densities defined by these maps are uniformly integrable in $L^p(\dspace)$, and $\dspace$ is compact,
\begin{equation}\label{limit:lam_pf_D}
	\pfpriordensa(\qoia)\to \pfpriordens(\qoi) \text{ in } L^p(\pspace).
\end{equation}
\end{theorem}

\begin{proof}
Let $\epsilon>0$.
By Lemma~\ref{lem:lipschitz_approx}, there exists a Lipschitz continuous $\widetilde{\pfpriordens}$ approximating $\pfpriordens$ such that
\begin{equation}\label{eq:poly_approx_error_D}
\norm{\pfpriordens(q) - \widetilde{\pfpriordens}(q)}_{L^p(\dspace)} < \frac{\epsilon}{4}.
\end{equation}
Applying the triangle inequality three times gives
\begin{align}\label{eq:pf_error_4ae_D}
\norm{\pfpriordensa(\qoia)-\pfpriordens(\qoi)}_{L^p(\pspace)} &\leq \norm{\pfpriordensa(\qoia) - \pfpriordens(\qoia)}_{L^p(\pspace)} \nonumber \\
 &\indent + \norm{\pfpriordens(\qoia) - \widetilde{\pfpriordens}(\qoia)}_{L^p(\pspace)}\nonumber\\
&\indent + \norm{\widetilde{\pfpriordens}(\qoia)-\widetilde{\pfpriordens}(\qoi)}_{L^p(\pspace)} \nonumber \\
&\indent + \norm{\widetilde{\pfpriordens}(\qoi) - \pfpriordens(\qoi)}_{L^p(\pspace)}.
\end{align}
Recalling that $\widetilde{\pfpriordens}$ is Lipschitz continuous, there is a constant $C>0$ such that
\begin{equation*}
	\norm{\widetilde{\pfpriordens}(\qoia)-\widetilde{\pfpriordens}(\qoi)}_{L^p(\pspace)} \leq C \norm{\qoia - \qoi}_{L^p(\pspace)}.
\end{equation*}
Then, $\qoia\to\qoi$ in $L^p(\pspace)$ implies that the third term on the right-hand side of~\eqref{eq:pf_error_4ae_D}  is bounded by $\epsilon/4$ by setting $n$ sufficiently large.

The first, second, and fourth terms on the right-hand side of~\eqref{eq:pf_error_4ae_D} are equivalently written, respectively, as
\begin{align*}
\norm{\pfpriordensa(\qoia) - \pfpriordens(\qoia)}_{L^p(\pspace)} &= \norm{\pfpriordensa(q) - \pfpriordens(q)}_{L^p(\dspace)}, \\
 \norm{\pfpriordens(\qoia) - \widetilde{\pfpriordens}(\qoia)}_{L^p(\pspace)} &=  \norm{\pfpriordens(q) - \widetilde{\pfpriordens}(q)}_{L^p(\dspace)},\\
\norm{\widetilde{\pfpriordens}(\qoi) - \pfpriordens(\qoi)}_{L^p(\pspace)} &= \norm{\widetilde{\pfpriordens}(q) - \pfpriordens(q)}_{L^p(\dspace)}.
\end{align*}
By~\eqref{eq:poly_approx_error_D}, the second and fourth term on the right-hand side of \eqref{eq:pf_error_4ae_D} are bounded by $\epsilon/4$.
Finally, by Lemma~\ref{lem:pf_prior_almostauec}, the first term on the right-hand side of~\eqref{eq:pf_error_4ae_D} is bounded by $\epsilon/4$, which proves~\eqref{limit:lam_pf_D}.
\end{proof}


The following corollary is included for completeness and states what happens if the conditions in Assumption~\ref{assump:almost_aec} are strengthened to hold a.e.~instead of in an almost sense.
The result is a strengthening of the conclusions in Lemma~\ref{lem:pf_prior_almostauec_1}, not requiring Lemma~\ref{lem:pf_prior_almostauec} (the uniform integrability in $L^p$ is now a given), while the conclusion of Theorem~\ref{thm:lam_pf_D} remains unchanged.
The changes to the proofs given above are straightforward and are therefore omitted.

\begin{corollary}\label{cor:pf_convergence}
Suppose $1\leq p<\infty$ and $\qoia\to \qoi$ in $L^p(\pspace)$.
If the criteria in Assumption~\ref{assump:almost_aec} hold a.e.~instead of in an almost sense, then (1) $\pfpriordens(q)$ is a.e.~continuous on $\dspace$, and (2)
\begin{equation}\label{limit:ptwise_pf}
	\pfpriordensa(q) \to \pfpriordens(q)  \indent \text{for a.e. } q\in\dspace.
\end{equation}
Furthermore, for any compact subset $D_c\subset \dspace$ and $1\leq r\leq \infty$,
\begin{equation}\label{limit:q_pf}
	\pfpriordensa(q)\to \pfpriordens(q) \indent \text{in }  L^r(D_c).
\end{equation}
Moreover, if $\dspace$ is compact,
\begin{equation}\label{limit:lam_pf}
	\pfpriordensa(\qoia)\to \pfpriordens(\qoi) \indent \text{in } L^p(\pspace).
\end{equation}
\end{corollary}

\section{Inverse problem analysis}\label{sec:inverse-theory}
In this section we analyze the convergence of probability density functions obtained from solving the inverse problem using approximate maps that converge to the exact map in the $L^p$ sense where $p$ is fixed and $1\leq p <\infty$.

\subsection{Problem definition and solution}
We begin by defining and summarizing the inverse problem and its solution considered in this work.
For a more thorough discussion on this inverse problem and the theory of existence, uniqueness, and stability of solutions, we direct the interested reader to \cite{BJW18a}, which also includes a comparison to alternative formulations and solutions of inverse problems including a specific comparison to the popular Bayesian formulations of inverse problems  \cite{Stuart_AN_2010, Kaipio_S_JCAM_2007, Berger_TEST_1994, Bernardo1994, Robert2001, Gelman2013, Jaynes1998}.
\begin{definition}[Inverse Problem and Consistent Measure]\label{def:inverse-problem}
Given a probability measure $\obsmeas$ on $(\dspace, \dborel)$ that is absolutely continuous with respect $\dmeas$ and admits a density $\obsdens$,
the inverse problem is to determine a probability measure $P_\pspace$ on $(\pspace, \pborel)$ that is absolutely continuous with respect to $\pmeas$ and admits a probability density
$\pi_\pspace$,
such that the subsequent push-forward measure induced by the map, $Q(\lambda)$, satisfies
\begin{equation}\label{eq:invdefn}
P_\pspace(Q^{-1}(A)) = P^{Q}_\dspace(A) = \obsmeas(A),
\end{equation}
for any $A\in \dborel$.
We refer to any probability measure $P_\pspace$ that satisfies \eqref{eq:invdefn} as a {\bf consistent} solution to the inverse problem.
\end{definition}

Given an {\em \cbprior} density, we make the following predictability assumption based on the push-forward densities of the initial density that ensures the solvability of the inverse problem using either the approximate or exact QoI maps.

\begin{assumption}\label{assump:dom}
There exists $C>0$ such that for any $n\in\mathbb{N}$ and a.e.~$q\in \dspace$, $\obsdens(q)\leq C\pfpriordens(q)$ and $\obsdens(q)\leq C\pfpriordensa(q)$.
\end{assumption}

This assumption is a modified form of the predictability assumptions previously used in \cite{BJW18a} and \cite{BJW18b}.
This is referred to as a predictability assumption because it ensures that QoI data that are likely to be observed are also likely to be predicted by the push-forward densities associated with the initial density.
It is important to note that this is an assumption on {\em both} the initial density and the QoI maps.

From a computational perspective, this form of the predictability assumption ensures that rejection sampling is a viable numerical approach to generate independent identically distributed (iid) sets of samples from a consistent solution to the inverse problem as established in \cite{BJW18a}.
Specifically, we can first generate a set of iid~samples from the initial density and then propagate these samples to the output space using any fixed QoI map.
This defines a set of {\em proposal samples} in the output space.
Then, we can perform rejection sampling using the {\em target} observed density.
This defines a set of accepted samples in the output space along with a corresponding set of parameter samples.
The accepted samples in the output space are iid~samples from the observed density.
Subsequently, the corresponding set of parameter samples are iid~samples from a consistent solution defined as the {\em updated density} to emphasize how the information from the QoI map has effectively updated the initial density.
This updated density has a closed form expression that follows by an application of the Disintegration Theorem \cite{Dellacherie_Meyer} for a QoI map satisfying a predictability assumption.
\begin{theorem}[Existence and Uniqueness of Solutions~\cite{BJW18a}]\label{thm:consistent_posterior}
Given an \cbprior~probability measure $\priormeas$ on $(\pspace,\pborel)$ with an associated density $\priordens$ and a QoI map $Q$ that satisfies a predictability assumption, the updated probability measure $\postmeas$ on $(\pspace, \pborel)$ defined by
\begin{equation}\label{eq:consistent_posterior}
	\postmeas(A) = \int_{\dspace} \bigg(  \int_{A\cap Q^{-1}(q)} \, \priordens(\lambda)\frac{\obsdens(Q(\lambda))}{\pfpriordens(Q(\lambda))} d\mu_{\pspace,q}(\lambda) \bigg) \, d\dmeas(q), \ \forall A\in\pborel
\end{equation}
is a consistent solution to the inverse problem in the sense of \eqref{eq:invdefn} and is uniquely determined for a given \cbprior~probability measure $\priormeas$ on $(\pspace,\pborel)$.
Here, $\mu_{\pspace,q}$ denotes the disintegrated measure of $\pmeas$\footnote{For those unfamiliar with disintegrations of measures, it is helpful to think of this like a nonlinear version of Fubini's theorem.}.
\end{theorem}

The updated density to the exact QoI map is then identified as
\begin{equation}\label{eq:postpdf}
\postdens(\lambda) = \priordens(\lambda)\frac{\obsdens(Q(\lambda))}{\pfpriordens(Q(\lambda))}, \quad \lambda \in \Lambda,
\end{equation}
which we often write as
\begin{equation}\label{eq:postpdfr}
	\postdens(\lambda) = \priordens(\lambda)r(\lambda), \quad \text{with} \ r(\lambda) := \frac{\obsdens(\qoi)}{\pfpriordens(\qoi)}.
\end{equation}
Here, the ratio denoted by $r(\lambda)$ has a practical interpretation as a {\em rejection ratio} for a randomly generated sample from $\priordens$.

Similarly, for any of the approximate QoI maps that satisfy the predictability assumption, the approximate updated densities are expressed as
\begin{equation}\label{eq:postpdfa}
\postdensa(\lambda) := \priordens(\lambda)\frac{\obsdens(\qoia)}{\pfpriordensa(\qoia)} = \priordens(\lambda)\ra, \quad \text{with} \ \ra := \frac{\obsdens(\qoia)}{\pfpriordensa(\qoia)}.
\end{equation}
The error in the total variation metric of the approximate \cbposterior~density is given by
\begin{equation}\label{eq:pdf_error}
	\norm{\postdensa(\lambda) - \postdens(\lambda)}_{L^1(\pspace)}  = \int_\pspace \priordens(\lambda)\abs{\ra - r(\lambda)}\, d\pmeas = \int_\pspace \abs{\ra - r(\lambda)} \, d\priormeas.
\end{equation}
The main takeaway from the inverse analysis of \cite{BJW18b} is that the total variation error goes to zero if the approximate QoI maps converge to the exact QoI map in the $L^\infty$ sense.
In this work, we assume the approximate maps converge in the $L^p$ sense and show that the approximate updated densities converge in the $L^p$ sense.

\subsection{Convergence analysis}

The following lemma is interpreted as providing a specific form for the subset of the data space on which the family of probability measures defined by the exact push-forward and the tail-end of the approximate push-forward probability measures (defined by their densities) is considered tight.

\begin{lemma}\label{lem:existence_of_constant}
Suppose $1\leq p<\infty$ and $\qoia\to\qoi$ in $L^p(\pspace)$.
If Assumption~\ref{assump:almost_aec} holds, then for any $\delta>0$, there exists $a>0$, compact $D_a\in\dborel$ and $N>0$ such that for any $n>N$
\begin{align}
\pfpriordens(q) > a,&\  \pfpriordensa(q) > a,\ \ \  \forall q\in D_a \label{eq:constant_bd1}\\
\int_{D_a} \pfpriordens(q)\, d\dmeas > 1-\delta &\text{ , } \int_{D_a} \pfpriordensa(q)\, d\dmeas > 1-\delta \label{eq:constant_bd2}
\end{align}
\end{lemma}

Before we prove this lemma,  we recall a few standard results from measure theory to simplify the first few steps of the proof.
First, if $(X,\mathcal{B}_X,\mu_X)$ is a measure space and  $f\in L^1(X)$, then for any $\epsilon>0$ there exists $a>0$ such that
\begin{equation*}
	\int_{\set{x\, : \, \abs{f(x)}>a}} \abs{f(x)}\, d\mu_X > \int_X \abs{f(x)}\, d\mu_X - \epsilon.
\end{equation*}
Moreover, if the measure space is $\sigma$-finite, then there exists $A\in\mathcal{B}_X$ such that $\mu_X(A)<\infty$ and
\begin{equation*}
	\int_A \abs{f(x)}\, d\mu_X > \int_X \abs{f(x)}\, d\mu_X - \epsilon.
\end{equation*}
Combining these results, it is possible to choose $a>0$ and $A$ compact such that
\begin{equation*}
	\int_A \abs{f(x)}\, d\mu_X > \int_X \abs{f(x)}\, d\mu_X - \epsilon, \text{ and } \abs{f(x)} > a \ \forall x\in A.
\end{equation*}
We also make use of the standard measure theory results involving approximating sets of finite-measure with open or compact sets as discussed following Lemma~\ref{lem:pf_prior_almostauec_1}.

\begin{proof}
Let $0<\delta<1$.
Following the discussion above, there exists $a>0$ and compact $D\in \dborel$ (so $\dmeas(D)<\infty$) such that
\begin{equation*}
\int_{D} \pfpriordens(q)\, d\dmeas>1-\frac{\delta}{4}, \text{ and } \pfpriordens(q) > 2a, \  \forall q\in D.
\end{equation*}
Using Assumption~\ref{assump:almost_aec}, choose $\eta>0$ sufficiently small and $N_{\eta}$ an open set such that the sequence of approximate push-forwards is a.e.c.~on $D_a:= D\backslash N_{\eta}$, $\dmeas(D_a) > \dmeas(D) - \eta$, and
\begin{equation*}
\int_{D_a} \pfpriordens(q)\, d\dmeas>1-\frac{\delta}{2}, \text{ and } \pfpriordens(q) > 2a, \  \forall q\in D_a.
\end{equation*}

Let $0< \epsilon<\min\{a, \frac{\delta}{2\dmeas(D_a)}\}$.
By design, $D_a$ is itself compact, so the sequence of approximate push-forwards is in fact a.u.e.c.~on $D_a$.
Thus, by application of Theorem 2 in \cite{Sweeting_86} on $D_a$ there exists $N>0$ such that for any $n>N$ and $q\in D_a$,
\[
-\epsilon < \pfpriordensa(q) - \pfpriordens(q) <\epsilon.
\]
Then, $q\in D_a$ implies $\pfpriordensa(q) > \pfpriordens(q) - \epsilon > a$, which proves \eqref{eq:constant_bd1}.

Finally, for any $n>N$,
\begin{align*}
\int_{D_a} \pfpriordensa(q)\, d\dmeas &= \int_{D_a} \left[\pfpriordensa(q)-\pfpriordens(q)\right]\, d\dmeas + \int_{D_a} \pfpriordens(q)\, d\dmeas \\
&> -\epsilon\dmeas(D_a) + 1-\frac{\delta}{2} \\
&> 1- \delta.
\end{align*}
This proves \eqref{eq:constant_bd2}.
\end{proof}

In actuality, a weaker conclusion than this lemma provides is needed to prove Theorem~\ref{thm:posterior_convergence} below.
Specifically, it is sufficient to prove that there exists a sequence of subsets in the data space associated with each $\pfpriordensa$ containing ``most'' of the probability for each density while simultaneously bounding each density from below.
However, the existence of a common data set not only serves to simplify the notation in the proof of the following theorem (we can avoid applying a subscript $n$ to sets), but it also provides a useful conceptualization of this set in terms of a tightness property of the densities.
This common set also implies a type of ``effective support'' containing ``most of the predicted probability'' for both the exact push-forward density and the tail-end of the sequence of approximate push-forward densities.
With this lemma in-hand, we now state the main result of this section giving the convergence of approximate updated densities in an $L^p$ sense when the approximate maps converge in an $L^p$ sense.

\begin{theorem}[Convergence of \cbposterior~densities]\label{thm:posterior_convergence}
Suppose $1\leq p<\infty$ and $\qoia\to \qoi$ in $L^p(\pspace)$.
If $\priordens\in L^p(\pspace)$, $\obsdens\in L^p(\dspace)$, Assumptions~\ref{assump:almost_aec} and \ref{assump:dom} hold,  and the family of push-forward densities defined by these maps are uniformly integrable in $L^p(\dspace)$, then
	\begin{equation}\label{eq:post_conv}
	\postdensa \to \postdens \ \text{ in } L^p(\pspace).
	\end{equation}
\end{theorem}
\indent Without loss of generality, we prove this theorem under some additional simplifying assumptions.
Specifically, we assume that $\obsdens$ and $\priordens$ are both Lipschitz continuous and that $\dspace$ and $\pspace$ are both compact.
If this is not the case, we can carry out the analysis using ``sufficiently good'' approximations to $\obsdens$ and $\priordens$ that are Lipschitz continuous with compact support by Lemma~\ref{lem:lipschitz_approx}, and simply use triangle inequalities to prove the result for more general initial and observed densities and non-compact parameter and data spaces.

\begin{proof}
	Let $\epsilon>0$. Since $\priordens(\lambda)\in C(\pspace)$ with compact support, there exists $M>0$ such that for any $\lambda\in\pspace$
	\[
	\abs{\priordens(\lambda)} < M.
	\]
	Set $0<\delta< \frac{\epsilon^p}{2^{p+1}C^pM^{p-1}}$ (here, the $C$ is from Assumption~\ref{assump:dom}).
	By Lemma~\ref{lem:existence_of_constant}, there exists $a>0$, $D_a\in\dborel$ and $N_1>0$ such that for any $n>N_1$
	\begin{align*}
	\pfpriordens(q) > a, &\hspace{0.3cm} \pfpriordensa(q) > a, \hspace{0.3cm} \forall q\in D_a\\
	\int_{D_a} \pfpriordens(q)\, d\dmeas > 1-\delta, &\hspace{0.3cm} \int_{D_a} \pfpriordensa(q)\, d\dmeas > 1-\delta.
	\end{align*}
	Denote $\pspace_{a,n}=Q_n^{-1}(D_a)$, $\pspace_a = Q^{-1}(D_a)$.
	Then, by linearity of the integral operator,
	\begin{align*}
	\norm{\postdensa(\lambda) - \postdens(\lambda)}^p_{L^p(\pspace)} &= \int_{\pspace}  \abs{\postdensa(\lambda) - \postdens(\lambda)}^p \, d\pmeas\\
	&= \int_{\pspace_{a,n}}  \abs{\postdensa(\lambda) - \postdens(\lambda)}^p \, d\pmeas +  \int_{\pspace \backslash \pspace_{a,n}}  \abs{\postdensa(\lambda) - \postdens(\lambda)}^p \, d\pmeas \\
	&= \int_{\pspace_{a,n}} (\priordens(\lambda))^p \abs{r_n(\lambda) - r(\lambda)}^p \, d\pmeas +  \int_{\pspace \backslash \pspace_{a,n}} ( \priordens(\lambda))^p \abs{r_n(\lambda) - r(\lambda)}^p \, d\pmeas.
	\end{align*}
	Observe that the difference in ratios given by $\abs{\ra-r(\lambda)}$ can be rewritten as
	\begin{equation*}
	\abs{\ra-r(\lambda)} = \abs{\frac{\obsdens(\qoia)\pfpriordens(\qoi) - \obsdens(\qoi)\pfpriordensa(\qoia)}{\pfpriordensa(\qoia)\pfpriordens(\qoi)}}.
	\end{equation*}
	Then, by adding and subtracting $\obsdens(\qoi)\pfpriordens(\qoi)$ in the numerator, this difference is decomposed as
	\begin{equation}\label{eq:ratio_diff_decomp}
	\abs{\ra-r(\lambda)} \leq  \underbrace{\abs{\frac{\obsdens(\qoia)-\obsdens(\qoi)}{\pfpriordensa(\qoia)}}}_{T_1(\lambda)} + \underbrace{\abs{\frac{\obsdens(\qoi)\left[\pfpriordens(\qoi)-\pfpriordensa(\qoia)\right]}{\pfpriordensa(\qoia)\pfpriordens(\qoi)}}}_{T_2(\lambda)}.
	\end{equation}
	Observe that
	\begin{align*}
	\int_{\pspace_{a,n}} (\priordens(\lambda))^p\abs{\ra - r(\lambda)}^p \, d\pmeas
	&\leq \int_{\dspace} \int_{\pspace_{a,n} \cap Q_{n}^{-1}(q)} (\priordens(\lambda))^p (T_1(\lambda) + T_2(\lambda))^p \, d\mu_{\pspace,q} d\dmeas \\
	&\leq \int_{\dspace} \int_{\pspace_{a,n} \cap Q_{n}^{-1}(q)} (\priordens(\lambda))^p \left(2\max_{\lambda\in\pspace}\{T_1(\lambda), T_2(\lambda)\}\right)^p \, d\mu_{\pspace,q} d\dmeas.
	\end{align*}
	Thus, we show that $ \int_{\pspace_{a,n}} (\priordens(\lambda))^p\abs{\ra - r(\lambda)}^p \, d\pmeas < \frac{\epsilon^p}{2}$ by proving
	\[
	\int_{\dspace} \int_{\pspace_{a,n} \cap Q_{n}^{-1}(q)} (\priordens(\lambda))^p (T_1(\lambda))^p \, d\mu_{\pspace,q} d\dmeas < \frac{\epsilon^p}{2^{p+2}},
	\]
	and
	\[
	\int_{\dspace} \int_{\pspace_{a,n} \cap Q_{n}^{-1}(q)} (\priordens(\lambda))^p (T_2(\lambda))^p \, d\mu_{\pspace,q} d\dmeas < \frac{\epsilon^p}{2^{p+2}}.
	\]

	Denote the Lipschitz constant for $\obsdens$ by $C_1\geq 0$, then
	\begin{equation}
	T_1(\lambda) \leq \frac{C_1\abs{\qoi-\qoia}}{\pfpriordensa(\qoia)}.
	\end{equation}
	H\"older's inequality then implies
	\begin{align*}
	& \indent \int_{\dspace} \int_{\pspace_{a,n} \cap Q_{n}^{-1}(q)} (\priordens(\lambda))^p (T_{1}(\lambda))^p \, d\mu_{\pspace,q}(\lambda) \, d\dmeas(q) \\
	&\leq (C_1)^p \int_{\dspace} \int_{\pspace_{a,n} \cap Q_{n}^{-1}(q)} \abs{\qoi - \qoia}^p \left( \frac{\priordens(\lambda)}{\pfpriordensa(\qoia)} \right)^p \, d\mu_{\pspace,q}(\lambda) \, d\dmeas(q) \\
	&\leq (C_1)^p \int_{\dspace} \norm{\abs{\qoi-\qoia}^p}_{L^{1}(\pspace_{a,n} \cap Q_{n}^{-1}(q))} \norm{\left(\frac{\priordens(\lambda)}{\pfpriordensa (\qoia)}\right)^p}_{L^{\infty}(\pspace_{a,n} \cap Q_{n}^{-1}(q))} \, d\dmeas(q) \\
	&\leq (C_1)^p \int_{\dspace} \norm{\qoi-\qoia}^p_{L^{p}(\pspace_{a,n} \cap Q_{n}^{-1}(q))} \norm{\left(\frac{\priordens(\lambda)}{\pfpriordensa (\qoia)}\right)^p}_{L^{\infty}(\pspace_{a,n} \cap Q_{n}^{-1}(q))} \, d\dmeas(q)\\
	&\leq (C_1)^p \norm{\left(\frac{\priordens(\lambda)}{\pfpriordensa (\qoia)}\right)^p}_{L^{\infty}(\pspace_{a,n})} \int_{\dspace} \norm{\qoi-\qoia}^p_{L^{p}(\pspace_{a,n} \cap Q_{n}^{-1}(q))}   \, d\dmeas(q)
	\end{align*}
	By the Disintegration Theorem,
	\begin{align*}
	\int_{\pspace_{a,n}} \abs{\qoi - \qoia}^p\, d\pmeas &= \int_{\dspace} \int_{\pspace_{a,n}\cap Q_n^{-1}(q)} \abs{\qoi - \qoia}^p d\mu_{\pspace,q}(\lambda) \, d\dmeas(q) \\
	&=\int_{\dspace} \norm{\qoi-\qoia}^p_{L^{p}(\pspace_{a,n} \cap Q_{n}^{-1}(q))}\, d\dmeas(q).
	\end{align*}
	Then, we have
	\begin{align*}
	& \indent \int_{\dspace} \int_{\pspace_{a,n} \cap Q_{n}^{-1}(q)} (\priordens(\lambda))^p (T_{1}(\lambda))^p \, d\mu_{\pspace,q}(\lambda) \, d\dmeas(q) \\
	&\leq (C_1)^p \norm{\left(\frac{\priordens(\lambda)}{\pfpriordensa (\qoia)}\right)^p}_{L^{\infty}(\pspace_{a,n})} \int_{\pspace_{a,n}} \abs{\qoi - \qoia}^p\, d\pmeas\\
	&\leq (C_1)^p \norm{\left(\frac{\priordens(\lambda)}{\pfpriordensa (\qoia)}\right)^p}_{L^{\infty}(\pspace_{a,n})}  \norm{\qoi-\qoia}^p_{L^{p}(\pspace)}.
	\end{align*}
	By construction,
	\[
	\norm{\left( \frac{\priordens(\lambda)}{\pfpriordensa (\qoia)} \right)^p }_{L^{\infty}(\pspace_{a,n})}  \leq \frac{M^p}{a^p}.
	\]
	Since $\qoia\to \qoi$ in $L^p(\pspace)$, there exists $N_2>0$ such that for any $ n>N_2$,
	\[
	\norm{\qoia - \qoi}^p_{L^p(\pspace)} < \frac{a^p\epsilon^p}{2^{p+2} (C_1)^p M^p}.
	\]
	Combining these inequalities gives
	\[
	\int_{\dspace} \int_{\pspace_{a,n} \cap Q_{n}^{-1}(q)} (\priordens(\lambda))^p (T_{1}(\lambda))^p \, d\mu_{\pspace,q}(\lambda) \, d\dmeas(q) < \frac{\epsilon^p}{2^{p+2}}.
	\]
	We now bound the term involving $T_2(\lambda)$.
	First, rewrite Assumption~\ref{assump:dom} as $\displaystyle \frac{\obsdens(q)}{\pfpriordens(q)} \leq C$ for a.e.~$q\in \dspace$.
	Then, use H\"older's inequality and the disintegration theorem as before to get
	\begin{align*}
	& \indent \int_{\dspace} \int_{\pspace_{a,n} \cap Q_{n}^{-1}(q)} (\priordens(\lambda))^p (T_{2}(\lambda))^p \, d\mu_{\pspace,q}(\lambda) \, d\dmeas(q) \\
	&\leq C^p \int_{\dspace} \int_{\pspace_{a,n} \cap Q_{n}^{-1}(q)} \left[\pfpriordens(\qoi) - \pfpriordensa(\qoia) \right]^p \left( \frac{\priordens(\lambda)}{\pfpriordensa(\qoia)} \right)^p \, d\mu_{\pspace,q}(\lambda) \, d\dmeas(q) \\
	&\leq C^p \int_{\dspace} \norm{\pfpriordens(\qoi)-\pfpriordensa(\qoia)}^p_{L^{p}(\pspace_{a,n} \cap Q_{n}^{-1}(q))} \norm{\left(\frac{\priordens(\lambda)}{\pfpriordensa (\qoia)}\right)^p}_{L^{\infty}(\pspace_{a,n} \cap Q_{n}^{-1}(q))} \, d\dmeas(q)\\
	&\leq C^p  \norm{\left(\frac{\priordens(\lambda)}{\pfpriordensa (\qoia)}\right)^p}_{L^{\infty}(\pspace_{a,n})}  \norm{\pfpriordens(\qoi)-\pfpriordensa(\qoia)}^p_{L^{p}(\pspace)}.
	\end{align*}
	By \eqref{limit:lam_pf_D} of Theorem~\ref{thm:lam_pf_D}, $\norm{\pfpriordens(\qoi) - \pfpriordensa(\qoia)}_{L^p(\pspace)} \to 0$. 
	It follows in a similar manner as before that there exists $N_3>0$ such that for any $n>N_3$,
	\[
	\int_{\dspace} \int_{\pspace_{a,n} \cap Q_{n}^{-1}(q)} (\priordens(\lambda))^p (T_{2}(\lambda))^p \, d\mu_{\pspace,q}(\lambda) \, d\dmeas(q) < \frac{\epsilon^p}{2^{p+2}}.
	\]
	Set $N=\max\{N_1, N_2, N_3\}$.
	For any $n>N$,
	\begin{equation*}
	\int_{\pspace_{a,n}} (\priordens(\lambda))^p \abs{r_n(\lambda) - r(\lambda)}^p \, d\pmeas < \frac{\epsilon^p}{2}.
	\end{equation*}
	%

	We now bound the other error term defined on the set $\pspace\backslash\pspace_{a,n}$.
	\begin{align*}
	\int_{\pspace \backslash \pspace_{a,n}}  (\priordens(\lambda))^p \abs{r_n(\lambda) - r(\lambda)}^p \, d\pmeas
	&\leq \int_{\pspace \backslash \pspace_{a,n}} \priordens(\lambda) (2C)^pM^{p-1} \, d\pmeas\\
	&= (2C)^pM^{p-1} \int_{\dspace\backslash D_a} \pfpriordensa(q)\, d\dmeas \\
	&< (2C)^pM^{p-1}\delta \\
	&< \frac{\epsilon^p}{2}.
	\end{align*}

	Thus, the sum of the error terms for any $ n>N$ satisfies
	\begin{align*}
	\norm{\postdensa(\lambda) - \postdens(\lambda)}^p_{L^p(\pspace)}
	&= \int_{\pspace_{a,n}}  \abs{\postdensa(\lambda) - \postdens(\lambda)}^p \, d\pmeas +  \int_{\pspace \backslash \pspace_{a,n}}  \abs{\postdensa(\lambda) - \postdens(\lambda)}^p \, d\pmeas\\
	&< \frac{\epsilon^p}{2} + \frac{\epsilon^p}{2} \\
	&= \epsilon^p.
	\end{align*}
	Raising each side to the $1/p$ power finishes the proof.
\end{proof}

\section{Numerical Considerations, Computational Estimates, and Impact on Assumptions}

\subsection{Finite sampling and density estimation}\label{sec:other-errors}

We first discuss how estimating the push-forward of an initial density (for either the exact or approximate QoI maps) using straightforward finite sampling techniques impacts the solutions to the forward and inverse problems considered in this work.
Specifically, suppose that we first generate a finite set of independent identically distributed (iid) samples from the initial density.
Then, propagating this sample set through the QoI map constructs an iid sample set from the (unknown) push-forward density.
While this sample set comes from the correct push-forward density (relative to the map used), we ultimately perform analysis on the density.
In some cases, kernel density estimation (KDE) techniques work well for estimating the push-forward densities especially when $\dspace$ is low-dimensional (e.g., see~\cite{BJW18b} and the references therein).
When $\dspace$ is high-dimensional and the number of iid samples we can generate from the push-forward is limited (e.g., due to a computationally expensive QoI map), we may instead opt for parametric estimations of the density.
Whatever density estimation scheme is used introduces an error in both the solution to the forward problem and subsequently in the solution to the inverse problem.
In \cite{BJW18b}, we analyze the impact of this error on the forward and inverse solutions when the QoI maps converge in $L^\infty$.
The modifications to that analysis to use the $L^p$-norm considered in this work are mostly straightforward by substituting the appropriate KDE error bounds available in the literature.
Thus, for the sake of brevity, we omit this additional error analysis in this work.

\subsection{Assumptions: Context, Generality, and Verification}\label{sec:verifying-assumptions}

We now discuss the two main assumptions in this work primarily in the context of using numerical estimates of the approximate push-forward densities.
Below, we simply refer to these as the estimated densities to distinguish from the terminology of approximate push-forward densities previously used that are in fact ``exact'' relative to the associated approximate QoI maps.
For notational simplicity, we abuse notation and use $\pfpriordensa$ to denote the estimated densities.
For both the sake of simplicity and in light of Lemma~\ref{lem:lipschitz_approx}, we assume that Gaussian KDEs are used so that all numerical estimates of approximate densities are Lipschitz continuous.
This has an added conceptual and computational advantage, as we discuss below, in contextualizing the equicontinuity criterion of Assumption~\ref{assump:almost_aec}.
Specifically, a well-known result is that if a sequence of Lipschitz continuous functions has a common Lipschitz constant, then the sequence of functions is equicontinuous.

\subsubsection{Assumption~\ref{assump:almost_aec}}\label{sec:verifyassump1}

Recall that Assumption~\ref{assump:almost_aec} involved two criteria: uniform boundedness and an asymptotic notion of equicontinuity.
The assumption is that these criteria hold in an almost sense, which is rather permissive and allows for the theory to apply even in the presence of countably infinite numbers of singularities in push-forward densities.
For the interested reader, this is discussed further in Appendix~\ref{app:a} in the context of a simple example with a push-forward density exhibiting a singularity.
Thus, we generally expect that this assumption holds except for pathological examples.
Unfortunately, verifying any criterion holds only in an almost sense is not necessarily straightforward.
However, it is relatively straightforward to use the estimated densities to investigate if these criteria hold in the stronger a.e.~sense.
We first require some notation.
Denote by $B_{n,m}$ and $L_{n,m}$ the bound and Lipschitz constant, respectively, for the estimated density obtained from the $n$th approximate map using $m$ parameter samples.

Estimates of $(B_{n,m})$ may be obtained, for instance, by simply sorting the evaluation of density estimates associated with the $n$th map  at a set of $m$ iid samples used to form the estimated densities.
To obtain estimates of $(L_{n,m})$, one approach is to first compute linear combinations of gradients of the kernel used in the numerical estimates of the densities, restrict evaluation to the $m$ iid samples used to construct the estimate, take the norms of these gradients, and then finally apply a sorting algorithm.
For many kernels, such as the Gaussian one used in this work, gradients are easily determined using calculus.

As $n$ and $m$ are increased, convergence of $(B_{n,m})$ and $(L_{n,m})$ implies that the criteria of Assumption~\ref{assump:almost_aec} hold in an a.e.~sense.
This is demonstrated in the numerical examples of Section~\ref{sec:applications} where we use the function \verb|gaussian_kde| within the subpackage \verb|stats| of the library \verb|scipy| to estimate the densities and subsequently use the \verb|gradient| function within the library \verb|numpy| to estimate the derivatives of these estimated densities.
In Appendix~\ref{app:a}, we discuss one potential way to utilize the computations leading to (divergent) sequences of $(B_{n,m})$ and $(L_{n,m})$ to help improve accuracy in estimated densities when singularities are present, but we leave further in-depth investigation to future work.

\subsubsection{A conceptual example and Assumption~\ref{assump:dom}}\label{sec:verifyassump2}

From a measure-theoretic perspective, the so-called predictability assumption (i.e., Assumption~\ref{assump:dom}) ensures that the observed measure (defined by the observed density) is absolutely continuous with respect to the push-forward measures (defined by the push-forward densities) that are obtained using either the exact or approximate maps.
Absolute continuity of measures is equivalent to stating the existence of Radon-Nikodym derivatives (i.e., densities).
When this assumption holds, the updated measure and density exist and take the form given in~\eqref{eq:consistent_posterior}.
Since the updated density is in fact a density, its integral is in fact equal to one.
Note that this integral is easily manipulated as
\[
1= \int_{\pspace} \postdens(\lambda) \ d\pmeas = \int_{\pspace} \priordens(\lambda) r(Q(\lambda)) \ d\pmeas = \int_\pspace r(Q(\lambda)) \ d\priormeas = \mathbb{E}_i(r(\qoi)).
\]
Here, $\mathbb{E}_i$ indicates the expected value of a random variable with respect to the initial density.
To monitor if Assumption~\ref{assump:dom} is violated by any of the approximate maps, we compute Monte Carlo estimates of the value of $\mathbb{E}_i(r(\qoia))$ using the $m$ iid samples from the initial density used to construct the estimated push-forward densities.
This provides a very cheap and useful diagnostic tool~\cite{BJW18a} since values that are not close to unity indicate that Assumption~\ref{assump:dom} is not satisfied.
For the reader interested in building intuition about this diagnostic tool, see Appendix~\ref{app:a} involving a simple example with a singularity in the push-forward density so that values of $\mathbb{E}_i(r(\qoia))$ are influenced by both numerical errors and, in some cases, violations of the predictability assumption. 

\section{Numerical Examples}\label{sec:applications}


Surrogates of the QoI map can significantly reduce the computational cost of solving both forward and inverse UQ problems.
The popularity of using stochastic spectral methods, and polynomial chaos expansions (PCEs) in particular, to approximate QoI maps arising in UQ problems dates back several decades, e.g., see~\cite{GR-99, XK-02, MGK+04b, WK-06, MNR-07}.
The convergence theory of these methods dates back nearly a century starting with the homogeneous chaos introduced by Wiener~\cite{Wiener-38} with further analysis by Cameron and Martin~\cite{CM-47}.
However, the interested reader may find the far more recent reference \cite{XK-02} to be a more accessible introduction to this topic, and we simply note here that PCEs are understood to converge in an $L^2$ sense.
While the theory we developed applies to any sequence of approximate models that converge in an $L^p$ sense, we focus on numerical examples using PCEs given their prevalence in the literature.

Two numerical examples are considered showing application of the theory to commonly studied ordinary and partial differential equation models.
Moreover, we consider separate numerical approaches for generating the PCE approximations in each example.
The PCEs for the first example are extensively studied in \cite{XK-02} using an intrusive approach, which we choose to employ here for one particular PCE.
This allows the interested reader to modify the PCE used here by following the steps outlined in \cite{XK-02}.
The PCEs for the second example are obtained using a non-intrusive pseudo-spectral approach that exploits quadrature methods for accurate construction of the coefficients in the PCEs.
For more information on these approaches, we direct the interested reader to \cite{XK-02, GhanemSpanos, DebBabuskaOden} for traditional intrusive approaches, to~\cite{RNG+03, AZ-07} for non-intrusive approaches, and to \cite{2009arXiv0904.2040C} for a comparison of such different approaches. 

In both examples, we use standard Gaussian kernel density estimation (KDE) to approximate the push-forward density functions, which are subsequently used to estimate the errors in both the push-forward and updated densities.
As noted in Section~\ref{sec:other-errors}, the impact of the additional error arising from the use of KDEs is analyzed in  \cite{BJW18b} when the maps converge in $L^\infty$, but we leave the modifications of that analysis for $L^p$-norms to future work.

Two final remarks are in order.
First, while it is possible to form PCEs with respect to distributions unrelated to the initial distribution, this is in general not considered optimal (e.g., see~\cite{MNR-07} where such inefficiencies are explored in a Bayesian setting).
Thus, for both simplicity in computations and presentation, we choose to form the PCEs with respect to the initial distribution assumed for the parameters.
Second, the convergence of PCEs is in an $L^2$ sense with respect to the probability measure of the random variable used to define the orthogonal polynomials in the expansion.
Thus, we consider/interpret the initial distribution as the dominating measure in this case.

\subsection{ODE Example}\label{subsec:ode}

Consider the ordinary differential equation
\begin{equation}\label{eq:ode}
\frac{dy(t)}{dt} = - \lambda y, \indent y(0) = 1,
\end{equation}
where the decay rate coefficient $\lambda$ is treated as a random variable.
The QoI is taken as $Q(\lambda) = y(0.5, \lambda)$ where $y(t, \lambda)=e^{-\lambda t}$ is the exact solution of \eqref{eq:ode}, which has a clear dependence on the parameter value.
For this example, $\pspace = (-\infty, \infty)$ and $\dspace = (0,\infty)$. \\

\subsubsection{Convergence of forward problem}
Assume $\priordens\sim N(0,1)$.
Following \cite{XK-02}, the PCE of the solution $y(t,\lambda)$ is given by
\begin{equation*}
y(t, \lambda)=\sum\limits_{i=0}^{\infty} y_i(t, \lambda)\Phi_i(\zeta)
\end{equation*}
where $y_i$ denotes the $i$th coefficient of the expansion of $y(t,\lambda)$, and $\set{\Phi_i}$ denotes a complete orthogonal polynomial basis from the Askey scheme associated with a choice of random variable $\zeta$.
Since the initial distribution of $\lambda$ is assumed Gaussian, we choose $\set{\Phi_i}$ as the Hermite polynomials and $\zeta$ as a Gaussian distributed random variable for optimal convergence.
The approximate maps are then defined by
\begin{equation}\label{eq:pce_Q}
Q_n(\lambda) = \sum\limits_{i=0}^n y_i(0.5, \lambda)\Phi_i(\zeta)
\end{equation}
for $1\leq n\leq P$ where $P$ indicates the highest order polynomial in the expansion.
Using the intrusive approach detailed in \cite{XK-02}, a system of ordinary differential equations is formed and solved for each $n$ to obtain the coefficients $y_i(t, \lambda)$ for $0\leq i\leq P$, which are then evaluated at $t=0.5$ to obtain $Q_n(\lambda)$.
Reference results are obtained using the exact map defined by $Q(\lambda)=e^{-0.5\lambda}$ evaluated on the same set of parameter samples as the approximate maps to construct estimates of the exact push-forward and updated densities.

%

%

To numerically verify Assumption~\ref{assump:almost_aec}, we follow the approach discussed in Section~\ref{sec:verifyassump1} to compute sequences of bounds and Lipschitz constants.
Tables~\ref{tab:ODE_Bs} and \ref{tab:ODE_Ls} indicate convergence of both the computed bounds and Lipschitz constants as both $m$ and $n$ increase.
This suggests that both criteria of Assumption~\ref{assump:almost_aec} hold in the stronger a.e.~sense.
\begin{table}[htb!]
\begin{center}
\begin{tabular}{c|c|c|c|c|c} \hline
\diagbox{$m$}{$n$} & 1 & 2 & 3 & 4 & 5 \\ \hline
1E3 & 0.71   & 0.84   & 0.89   & 0.87  & 0.86  \\ \hline
1E4 & 0.71   & 0.87   & 0.90   & 0.87  & 0.86   \\ \hline
1E5 & 0.70   & 0.92   & 0.94   & 0.90  & 0.90 \\ \hline
\end{tabular}
\end{center}
\caption{Representative results for $(B_{n,m})$ for the estimated densities associated with the $n$th approximate map constructed from $m$ iid samples.}
\label{tab:ODE_Bs}
\end{table}
\begin{table}[htb!]
\begin{center}
\begin{tabular}{c|c|c|c|c|c} \hline
\diagbox{$m$}{$n$} & 1 & 2 & 3 & 4 & 5 \\ \hline
1E3 & 0.88   & 2.09    & 1.76    & 1.60  & 1.62  \\ \hline
1E4 & 0.84  & 4.07   & 2.49    & 2.14  & 2.21   \\ \hline
1E5 & 0.81   & 7.01   & 2.74    & 2.22  & 2.32 \\ \hline
\end{tabular}
\end{center}
\caption{Representative results for $(L_{n,m})$ for the estimated densities associated with the $n$th approximate map constructed from $m$ iid samples.}
\label{tab:ODE_Ls}
\end{table}

We now verify the convergence in Lemma~\ref{lem:pf_prior_almostauec_1} and Theorem~\ref{thm:lam_pf_D}.
First, a single initial set of $1E4$ iid~parameter samples drawn from the initial distribution are generated.
Then, the approximate push-forward densities are estimated using a standard Gaussian KDE applied to the approximate map evaluations of these parameter samples.
The effective support of the push-forward densities is contained within $D_c=[0,4]$.
We therefore compute Monte Carlo estimates of $	\|\pi_{\mathcal{D}}^Q(q)-\pi_{\mathcal{D}}^{Q_n}(q)\|_{L^r(D_c)} $ for various $r\geq 1$ using a fixed set of $1E4$ iid~uniform random samples of $D_c$.
The left plot in Figure~\ref{fig:ODE_forward_error} shows the corresponding error plots for $r=1,2,\ldots, 5$ as functions of the approximate map number.
These error plots indicate an exponential convergence of the approximate push-forward densities for $r\geq 1$, which verifies the convergence in Lemma~\ref{lem:pf_prior_almostauec_1}.
Before verifying the convergence in Theorem~\ref{thm:lam_pf_D}, we note that $\dspace=(0,\infty)$ is not actually compact, so there is no strict guarantee that the results from the theorem apply.
Nonetheless, the probability of any of QoI maps producing values in $\dspace$ greater than, say $10$, is nearly zero. 
In other words, the computational support of all probability densities on $\dspace$ for which it is plausible that a random number generator will sample from is effectively compact, and we expect the results for Theorem~\ref{thm:lam_pf_D} to hold. 
We estimate $\|\pi_{\mathcal{D}}^Q(Q(\lambda))-\pi_{\mathcal{D}}^{Q_n}(Q_n(\lambda))\|_{L^2(\Lambda)} $ using Monte Carlo estimates for each of the $n=1,2,\ldots, 5$ approximate maps and the same $1E4$ iid~samples drawn from the initial distribution as were used to construct the push-forwards.
These errors are summarized in the right plot of Figure~\ref{fig:ODE_forward_error} as a function of the approximate map number, which verifies the convergence in Theorem~\ref{thm:lam_pf_D}.
\begin{figure}[ht]
\begin{center}
\includegraphics[width=0.49\textwidth]{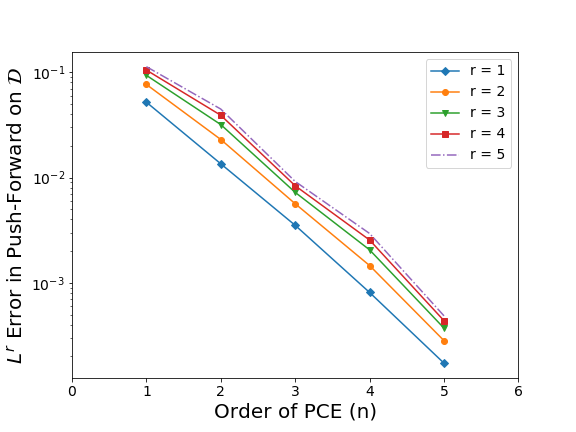}
\includegraphics[width=0.49\textwidth]{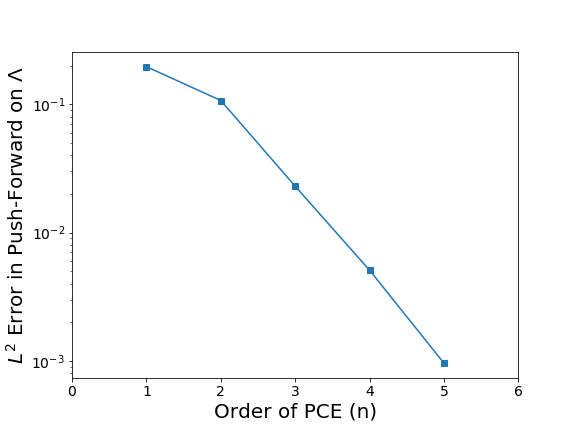}
\end{center}
\caption{Convergence of push-forward densities for ODE example.
  Left: $L^r$ error plot of the push-forward densities on $\dspace$ for $r=1, 2, \cdots, 5$.
   Right: $L^2$ error plot of the push-forward densities on $\pspace$.}
\label{fig:ODE_forward_error}
\end{figure}

\subsubsection{Convergence of inverse problem}
We now verify the convergence in Theorem~\ref{thm:posterior_convergence}.
Using the same initial distribution as above, we assume the observed density is given by a $N(1, 0.1^2)$ distribution.
We use the same approximate maps and approximate push-forward densities as above.
Before we analyze the convergence, we first verify the Assumption~\ref{assump:dom} using the expected value of the ratio, $\mathbb{E}_i(r(\qoia))$, mentioned in Section~\ref{sec:verifyassump2}.
\begin{table}[htb!]
\begin{center}
\begin{tabular}{c|c|c|c|c|c} \hline
$n$ & 1 & 2 & 3 & 4 & 5 \\ \hline
$\mathbb{E}_i(r(\qoia))$ & 0.58   & 0.98    & 0.98    & 0.98  & 0.98  \\ \hline
\end{tabular}
\end{center}
\caption{Representative estimates of $\mathbb{E}_i(r(\qoia))$ with $m=1E4$ iid samples used for both density estimation of associated push-forward densities and Monte Carlo estimate of $\mathbb{E}_i(r(\qoia))$.}
\label{tab:ODE_Ei}
\end{table}
Since approximate maps associated with $n\geq 2$ produce $\mathbb{E}_i(r(\qoia))$ estimates close to 1 from Table~\ref{tab:ODE_Ei}, it appears there is no violation of Assumption~\ref{assump:dom} for these maps.

\begin{figure}[ht]
\begin{center}
\includegraphics[width=0.5\textwidth]{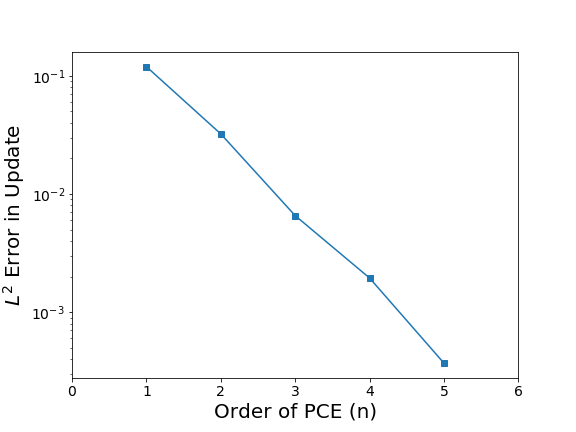}
\end{center}
\caption{Convergence of updated densities for ODE example in $L^2$.}
\label{fig:ODE_inverse_error}
\end{figure}

Figure~\ref{fig:ODE_inverse_error} illustrates the error in the updated densities associated with the $n$th approximate map. Here, the error is given by $\|\pi_{\Lambda}^{u,n}(\lambda)-\pi_{\Lambda}^u(\lambda)\|_{L^2(\Lambda)} $, which is again estimated using Monte Carlo integration with the same $1E4$ samples used to construct the push-forward densities. The plot demonstrates the $L^2$ convergence of the updated densities.

%


\subsection{PDE Example}\label{subsec:pde}
Consider the following PDE,
\begin{equation}\label{eq:pde}
\begin{cases}- \nabla\cdot(A\nabla u) = (e^{\lambda_1}\lambda_1^2\pi^2 + e^{\lambda_2}\lambda_2^2\pi^2)u, & \text{ in }\Omega \\
u = 0,  & \text{ on } \Gamma_0  \\
(A\nabla u)\cdot n = -e^{\lambda_2}\lambda_2\pi \sin(\lambda_1\pi x)\sin (\lambda_2\pi y) , &\text{ on } \Gamma_1  \\
(A\nabla u)\cdot n = e^{\lambda_2}\lambda_2\pi \sin(\lambda_1\pi x)\sin (\lambda_2\pi y) , & \text{ on } \Gamma_2 \\
(A\nabla u)\cdot n = e^{\lambda_1}\lambda_1\pi \cos(\lambda_1\pi x)\cos (\lambda_2\pi y) , &\text{ on } \Gamma_3   \\
\end{cases}
\end{equation}
where
\[ A = \begin{bmatrix} e^{\lambda_1} & 0 \\ 0 & e^{\lambda_2} \end{bmatrix} \]
$\Gamma_0 = \{ (x,y): x=0, 0\leq y\leq 1\}$, $\Gamma_1 = \{ (x,y): y=1, 0\leq x\leq 1\}$, $\Gamma_2 = \{ (x,y): y=0, 0\leq x\leq 1\}$, $\Gamma_3 = \{ (x,y): x=1, 0\leq y\leq 1\}$ and $\Omega = [0,1]\times [0,1]$.\\
The QoI is the average value of the solution $u$ in region $[a, b]\times [c, d]$,
\begin{equation}\label{eq:pde_qoi}
Q(\lambda_1,\lambda_2) =  \frac{1}{(b-a)(d-c)}\int_{c}^{d} \int_{a}^{b} u(x,y;\lambda_1,\lambda_2)\, dx\, dy
\end{equation}
where, in this example, we know $u(x,y;\lambda_1,\lambda_2)=\sin( \lambda_1\pi x )\cos (\lambda_2 \pi y)$ is the exact solution of \eqref{eq:pde} associated with a particular sample $(\lambda_1,\lambda_2)\in\pspace$, and we choose $a=c=0.4$, $b=d=0.6$.
For this example, $\pspace = (-\infty, \infty)\times (-\infty, \infty)$ and $\dspace =(-\infty,\infty)$. \\

%

\subsubsection{Convergence of forward problem}
As before, the goal of this section is to verify the convergence results of Lemma~\ref{lem:pf_prior_almostauec_1} and Theorem~\ref{thm:lam_pf_D}.
Assume $\lambda_1\sim N(\mu_1, \sigma_1^2)$ and $\lambda_2\sim N(\mu_2, \sigma_2^2)$ with $\mu_1=\mu_2=0$, $\sigma_1=\sigma_2=0.1$.
Then, we write the exact map $Q(\lambda_1,\lambda_2)$ as follows:
\[
	Q(\lambda_1, \lambda_2) = \sum_{i,j=0}^{\infty} q_{ij} \Phi_{ij}\left(\frac{\lambda_1 - \mu_1}{\sigma_1},\frac{\lambda_2 - \mu_2}{\sigma_2}\right).
\]
where $\set{\Phi_{ij}}$ (for $0\leq i,j<\infty$) denotes a complete 2-d orthogonal Hermite polynomial basis, $\set{q_{i,j}}$ (for $0\leq i,j<\infty$) denotes the corresponding coefficients of this PCE, and the equality is understood to hold in $L^2$.
Here, instead of formulating and solving a system of PDEs to compute the coefficients $\set{q_{ij}}$, we apply a non-intrusive pseudo-spectral approach to approximate the $\set{q_{ij}}$.
Specifically, starting with the PCE above, we compute (weighted) $L^2$-inner products of both sides with a fixed $\Phi_{ij}$.  Then, by exploiting the orthogonality of the polynomials, $q_{ij}$ is given as an integral of $Q(\lambda_1,\lambda_2)$ weighted by the product of $\Phi_{ij}$ and the underlying normal distribution for which these polynomials are orthogonal (and normalized by the weighted $L^2$-norm of $\Phi_{ij}$).
In general, we do not expect to have $Q(\lambda_1,\lambda_2)$ available in closed form to compute this integral, so we instead turn to quadrature methods to discretize the integrals over $\pspace$ defining $q_{ij}$ where numerical approximations of $Q(\lambda_1,\lambda_2)$ are used at each quadrature point.
To this end, we use a triangulation of a $50\times 50$ square mesh on $\Omega$ to obtain a numerical estimate of $u(x,y;\lambda_1,\lambda_2)$ and subsequently of $Q(\lambda_1,\lambda_2)$. 
Estimates of $Q(\lambda_1,\lambda_2)$ are obtained on a set of $400$ parameter samples taken from the tensor product of 20-point Gauss-hermite quadrature points in each dimension to ensure accuracy well beyond the degree of polynomials considered in this example.
Using these computed $q_{ij}$, we write the approximate maps as
\[
Q_n(\lambda_1, \lambda_2) = \sum_{i+j\leq n} q_{ij} \Phi_{ij}\left(\frac{\lambda_1 - \mu_1}{\sigma_1},\frac{\lambda_2 - \mu_2}{\sigma_2}\right).
\]
The reference results below are obtained by using the manufactured solution and evaluating a closed form expression for the resulting QoI on each of the parameter samples used by the approximate maps to construct estimates of the exact densities.

Before verifying assumptions and convergence of densities, we remark on some interesting behavior with regards to the QoI maps.
The closed form expression for the QoI map reveals an odd function in the $\lambda_1$-direction and an even function in the $\lambda_2$-direction, which is due to the oscillatory nature of the manufactured solution and domain $[a,b]\times[c,d]$ used to compute the QoI.
Subsequently, this implies the exact QoI map is orthogonal to many of the polynomials $\Phi_{ij}$.
In this example, we consider approximate maps of order $n=1,2,\ldots, 5$ where $\Phi_{ij}$ for any $i+j\leq n$ is given by the product $\Phi_i(\lambda_1)\Phi_j(\lambda_2)$.
Thus, it is straightforward to see that the exact QoI map is orthogonal in the weighted $L^2$-inner product to many of the Hermite polynomials.
In fact, the only $(i,j)$-pairs with $i+j\leq 5$ such that $q_{ij}$ is non-zero are
\[
	S:=\set{(1,0), (3,0), (1, 2), (5,0), (3,2), (1, 4)}.
\]
This is numerically confirmed as well where the quadrature method produces estimated values for $q_{ij}$ with $(i,j)\notin S$ that are orders of magnitude smaller than the coefficients associated with $(i,j)\in S$.
Consequently, the approximate maps $Q_1$, $Q_3$, and $Q_5$ are all distinct approximations to $Q$ whereas $Q_2$ and $Q_4$ are  almost indistinguishable from $Q_1$ and $Q_3$, respectively, with only slight variations present due to the use of numerical quadrature and the numerical solution of the PDE used in the quadrature method.
The impact of this on the convergence of densities is seen below. 


As before, we first verify Assumption~\ref{assump:almost_aec} as described in Section~\ref{sec:verifyassump1} to compute sequences of bounds and Lipschitz constants.
Tables~\ref{tab:PDE_Bs} and ~\ref{tab:PDE_Ls} indicate convergence of both the computed bounds and Lipschitz constants as both $m$ and $n$ increase.
As in the previous example, this suggests that both criteria of Assumpton~\ref{assump:almost_aec} hold in the stronger a.e.~sense.
\begin{table}[htb!]
\begin{center}
\begin{tabular}{c|c|c|c|c|c} \hline
\diagbox{$m$}{$n$} & 1 & 2 & 3 & 4 & 5 \\ \hline
1E3   & 2.63    & 2.63    & 2.61  & 2.61 & 2.61  \\ \hline
1E4   & 2.64    & 2.64    & 2.61  & 2.61  &  2.61\\ \hline
1E5   & 2.60   & 2.60    & 2.57  & 2.57 & 2.57 \\ \hline
\end{tabular}
\end{center}
\caption{Representative results for $(B_{n,m})$ for the estimated densities associated with the $n$th approximate map constructed from $m$ iid samples.}
\label{tab:PDE_Bs}
\end{table}
\begin{table}[htb!]
\begin{center}
\begin{tabular}{c|c|c|c|c|c} \hline
\diagbox{$m$}{$n$} & 1 & 2 & 3 & 4 & 5 \\ \hline
1E3   & 12.11    & 12.11    & 11.95  & 11.95  & 11.95 \\ \hline
1E4   & 11.47   & 11.47   & 11.41  & 11.41 & 11.41  \\ \hline
1E5   & 11.07   & 11.07    & 10.86  & 10.86 & 10.86 \\ \hline
\end{tabular}
\end{center}
\caption{Representative results for $(L_{n,m})$ for the estimated densities associated with the $n$th approximate map constructed from $m$ iid samples.}
\label{tab:PDE_Ls}
\end{table}

We now verify the convergence in Lemma~\ref{lem:pf_prior_almostauec_1} and Theorem~\ref{thm:lam_pf_D}.
A single initial set of $1E4$ iid~parameter samples are generated and each of the approximate maps are evaluated on these parameter samples to generate different approximate QoI sample sets.
Then, the push-forward densities for each approximate map are estimated using a standard Gaussian KDE.
The range of output samples falls within the interval $[-1,1]$, so we set $D_c=[-1,1]$ and estimate $	\|\pi_{\mathcal{D}}^Q(q)-\pi_{\mathcal{D}}^{Q_n}(q)\|_{L^r(D_c)} $ for various $r\geq 1$ using Monte Carlo estimates on a fixed set of $1E4$ uniform random samples in $D_c$.
\begin{figure}[ht]
\begin{center}
\includegraphics[width=0.49\textwidth]{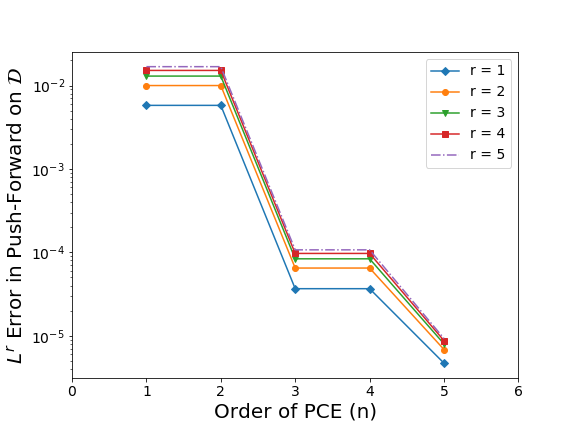}\hspace{0.1cm}
\includegraphics[width=0.49\textwidth]{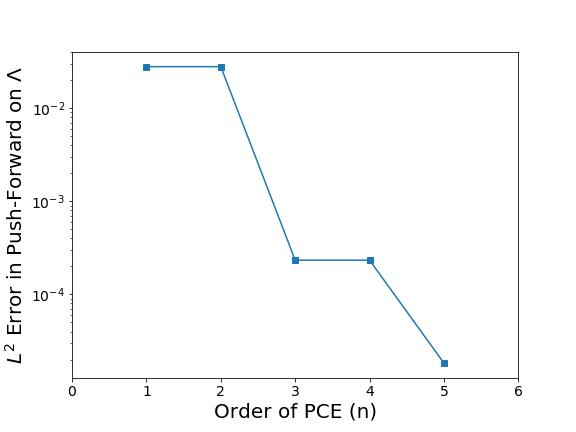}
\end{center}
\caption{Convergence of push-forward densities for PDE example.
  Left: $L^r$ error plot of the push-forward densities on $\dspace$ for $r=1, 2, \ldots, 5$.
   Right: $L^2$ error plot of the push-forward densities on $\pspace$.}
\label{fig:PDE_forward_error}
\end{figure}

The left plot of Figure~\ref{fig:PDE_forward_error} shows the corresponding error plots for $r=1,2,\ldots, 5$ as a function of the approximate map number.
Note that the error decreases when $n$ increases from an even to an odd integer but does not appear to change when $n$ increases from an odd to an even integer.
This is due to the symmetry of the QoI map and its orthogonality with respect to even-ordered polynomials as described above and is thus completely expected for this sequence of approximate maps.
As in the previous example, we note that $\dspace=(-\infty,\infty)$ so that the conclusions of Theorem~\ref{thm:lam_pf_D} cannot be guaranteed.
However, as before, virtually all samples are generated within a compact set of $\dspace$ (in this case $[-1,1]$), so we expect that to see convergence. 
We then estimate $\|\pi_{\mathcal{D}}^Q(Q(\lambda))-\pi_{\mathcal{D}}^{Q_n}(Q_n(\lambda))\|_{L^2(\Lambda)}$ using Monte Carlo estimates with the same $1E4$ samples used to estimate the approximate push-forwards for $n=1, 2,\ldots, 5$, which is summarized in the right plot of Figure~\ref{fig:PDE_forward_error} as a function of the approximate map number.
The overall trend of decreasing errors in both plots verifies the theoretical results of Lemma~\ref{lem:pf_prior_almostauec_1} and Theorem~\ref{thm:lam_pf_D}.

%


\subsubsection{Convergence of inverse problem}
We now verify the convergence in Theorem~\ref{thm:posterior_convergence}.
Using the same initial distribution as above, we assume the observed density is given by a $N(0.3, 0.1^2)$ distribution.
We use the same approximate maps and approximate push-forward densities as above.
Assumption~\ref{assump:dom} is again verfied using the expected value of the ratio, $\mathbb{E}_i(r(\qoia))$, mentioned in Section~\ref{sec:verifyassump2}.
For each of the approximate maps, rounding $\mathbb{E}_i(r(\qoia))$ at the second decimal gives $1.00$, so we omit summary of these values in a table.
The quality of these estimates of expected values is easily explained.
First, the QoI map and its approximations are all approximately linear within several standard deviations of the mean parameter values.
Second, a linear mapping of a Gaussian distribution produces a Gaussian distribution, and the use of $1E4$ samples along with a Gaussian KDE will in general produce excellent approximations to $1$-dimensional Gaussian distributions.
Thus, for this problem, all of the estimated push-forward densities are extremely accurate approximations to the push-forward densities associated to each approximate map over sets of high probability since they are approximately Gaussian.
When this occurs and the observed density is in the range of the QoI map, we generally expect values of $\mathbb{E}_i(r(\qoia))$ to be very close to $1.00$.

\begin{figure}[htb!]
\begin{center}
\includegraphics[width=0.5\textwidth]{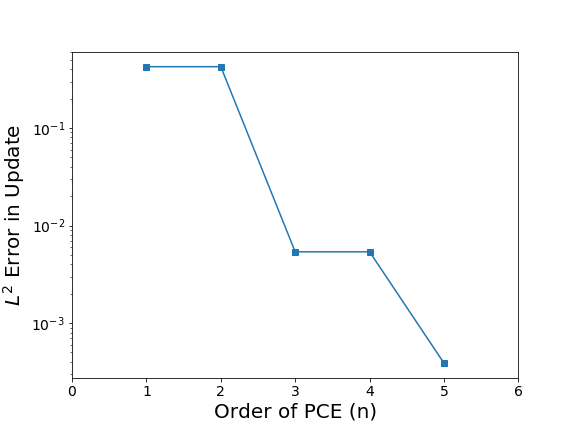}
\end{center}
\caption{Convergence of updated densities for PDE example in  $L^2$.}
\label{fig:PDE_inverse_error}
\end{figure}

Figure~\ref{fig:PDE_inverse_error} then illustrates the error of updated densities as a function of approximate map number.
Here, we again use Monte Carlo estimates of $\|\pi_{\Lambda}^{u,n}(\lambda)-\pi_{\Lambda}^u(\lambda)\|_{L^2(\Lambda)}$ using the same set of parameter samples used to construct the push-forwards, and we observe $L^2$ convergence of the updated densities.

%


\section{Conclusion}\label{sec:conclusions}
We developed a theoretical framework for analyzing the convergence of probability density functions computed using approximate models for both forward and inverse problems.
The theoretical results are quite general and apply to any $L^p$-convergent sequence of approximate models.
This greatly extends previous work that required almost uniform convergence of approximate models (i.e., $L^\infty$ convergence).
A simple numerical example producing a singular push-forward density is used to show the permissiveness of the two main assumptions used in this work.
Moreover, these assumptions are verified in each of the main numerical examples that demonstrate the convergence results for polynomial chaos expansions, which are commonly used to build approximate models in the literature.

\section{Acknowledgments}\label{sec:acknowledge}

T.~Wildey's work is supported by the Office of Science Early Career Research Program.
T.~Butler's and W.~Zhang's work is supported by the National Science Foundation under Grant No.~DMS-1818941.

\appendix

\section{An almost example and the assumptions}\label{app:a}

Here, we discuss the permissiveness of Assumption~\ref{assump:almost_aec} in the context of a simple example that highlights and builds intuition about several theoretical and computational points.

Let $\pspace = [-1,1]$.
Suppose the exact QoI map is given by $Q(\lambda)=\lambda^5$ and the approximate QoI maps, $(\qoia)$, are given by piecewise-linear interpolating splines using $n+2$ knots with $k$th knot given by $-1+\frac{2}{n+1}(k-1)$.
In other words, the $n$ refers to the number of regularly spaced interior points at which we evaluate the exact map to construct the approximate map.
It is clear that $\dspace=[-1,1]$ and $Q_n\to Q$ in $L^p(\pspace)$ for any $1\leq p\leq \infty$.

Now, suppose that the initial density is given by a uniform distribution on $\pspace$.
In this case, the exact push-forward density associated with the exact QoI map is given by
\[
	\pfpriordens(q) = \frac{1}{10}q^{-4/5},
\]
which is neither continuous on $\dspace$ nor in $L^\infty(\dspace)$.
In fact, $\pfpriordens \notin L^r(\dspace)$ for any $r\geq 5/4$.
The approximate push-forward densities are defined by a sequence of simple functions since they are defined by mapping a uniform density through piecewise-linear 1-to-1 maps.
When $n$ is odd, a straightforward computation shows that the $n$th approximate push-forward density is equal to the constant $\frac{(n+1)^4}{2^5}$ on the interval $\left(\frac{-2^5}{(n+1)^5},\frac{2^5}{(n+1)^5}\right)\subset\dspace$.
It follows that there are subsequential pointwise limits of these approximate push-forward densities evaluated at or near $q=0$ that either tend to infinity or can be made arbitrarily large.

While the features of the push-forward densities discussed above may seem problematic, there is no issue in applying the theory of this work to this problem.
To see this, note that for any $\epsilon>0$, there exists a uniform bound for the family of densities on $\dspace\backslash(-\epsilon/2,\epsilon/2)$.
Moreover, since the simple function approximate densities are bounded above on $\dspace\backslash(-\epsilon/2,\epsilon/2)$ and converge at a.e.~$q\in\dspace\backslash(-\epsilon/2,\epsilon/2)$, they converge almost uniformly and are subsequently almost a.e.c on $\dspace$.

The above example illustrates the generality of the theory developed in this work for forward UQ problems.
Specifically, having both criteria of Assumption~\ref{assump:almost_aec} hold in an almost sense is rather permissive since the theory applies to problems with push-forward densities containing countably infinite numbers of singularities.
Thus, under what we refer to as ``normal problem conditions'' where initial densities are non-singular and the QoI map possesses only a finite number of critical points (typically corresponding to the number of singularities in the push-forward density as seen in the above example), we  expect Assumption~\ref{assump:almost_aec} to hold.
However, a theoretical guarantee of convergence is not necessarily observed by the estimated densities if a computational budget restricts the quality of the estimates.
To ensure that computational estimates based on finite sampling are even remotely accurate, we consider a more practical restriction on the class of forward UQ problems for which we apply the theory.
Specifically, we restrict ourselves to problems that involve push-forward densities containing relatively small numbers of singularities.
This can generally be assured, for example, by restricting the class of initial densities to be non-singular and considering QoI maps that have finite numbers of critical points over the parameter space.

The locations of critical points of the maps and corresponding singularities in push-forward densities are typically not known a priori.
Yet, it is our ability to sufficiently sample around these points in the parameter and data spaces that most impact the point-wise accuracy of the estimated densities since most density estimation techniques (and especially Gaussian KDE) oversmooth ``peaks'' in densities.
Estimating locations of potential singularities may allow for significant improvement in estimated density accuracy over all of $\dspace$ using alternative sampling techniques.
While such numerical issues are not the focus of this work, we outline one computational approach to tackle this issue based on applying clustering techniques borrowed from machine learning to investigate the criteria of Assumption~\ref{assump:almost_aec} in the context of the above example.

Representative results of the quantitative analyses discussed in Section~\ref{sec:verifyassump1} are summarized in Tables~\ref{tab:concept_example_Bs} and \ref{tab:concept_example_Ls}.
\begin{table}[ht!]
\begin{center}
\begin{tabular}{c|c|c|c|c|c} \hline
\diagbox{$m$}{$n$} & 1 & 2 & 4 & 8 & 16 \\ \hline
1E3 & 0.54   & 1.52   & 2.50   & 2.94  & 3.19  \\ \hline
1E4 & 0.53   & 2.09   & 3.42   & 4.10  & 4.43   \\ \hline
1E5 & 0.51   & 3.03   & 4.44   & 5.63  & 6.12 \\ \hline
\end{tabular}
\end{center}
\caption{Representative results for $(B_{n,m})$ for the estimated densities associated with the $n$th approximate map constructed from $m$ iid samples.}
\label{tab:concept_example_Bs}
\end{table}
\begin{table}[ht!]
\begin{center}
\begin{tabular}{c|c|c|c|c|c} \hline
\diagbox{$m$}{$n$} & 1 & 2 & 4 & 8 & 16 \\ \hline
1E3 & 1.41   & 6.08    & 13.68    & 17.98  & 20.81  \\ \hline
1E4 & 2.19   & 13.89   & 27.26    & 38.35  & 44.06   \\ \hline
1E5 & 3.48   & 33.14   & 46.45    & 75.08  & 89.42 \\ \hline
\end{tabular}
\end{center}
\caption{Representative results for $(L_{n,m})$ for the estimated densities associated with the $n$th approximate map constructed from $m$ iid samples.}
\label{tab:concept_example_Ls}
\end{table}
In these tables, we observe the slow divergence of both the bounds and Lipschitz constants as both $m$ and $n$ increase.
While we omit further numerical analysis here, we comment on one potential approach to help improve accuracy of the estimated push-forward densities by identifying the regions in parameter space associated with data singularities. 
First, it may be possible to use the sorting algorithms on the arrays that returned the bounds and Lipschitz constants to identify points in data space where potential singularities are nearby. 
Subsequently, applying clustering algorithms on the corresponding samples in the parameter space, it may be possible to identify the regions in parameter space that should be sampled more extensively to improve the accuracy of the push-forward density estimates.

To give some intuition about the range of values we typically encounter for the diagnostic tool described in Section~\ref{sec:verifyassump2}, we return to the simple conceptual example.
We use three different observed densities, which we refer to as
\[
	\text{Case I: } \obsdens \sim N(0.50,0.1^2), \quad \text{Case II: } \obsdens \sim N(0.25,0.1^2), \quad \text{and Case III: } \obsdens\sim N(1,0.1^2).
\]
These cases demonstrate the variation in values we expect for $\mathbb{E}_i(r(\qoia))$ under different scenarios related to Assumption~\ref{assump:dom} and the magnitude of errors present in the estimated density relative to the effective support of the observed density.
In all cases, the effective support may be interpreted being within three standard deviations of the means of the observed densities.
Below, we refer to the plots in Figure~\ref{fig:concept_example_observed} for a visual reference of these cases that demonstrate the extent to which either Assumption~\ref{assump:dom} is violated or where significant errors in the estimated densities occur.
For simplicity in each of these plots, we only plot the exact push-forward density for the exact map, its estimation with $m=1E5$ iid samples, and the different observed densities.
\begin{figure}[htbp]
\begin{center}
\includegraphics[width=0.32\textwidth]{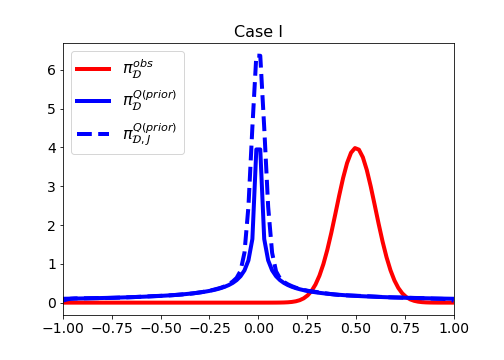}
\includegraphics[width=0.32\textwidth]{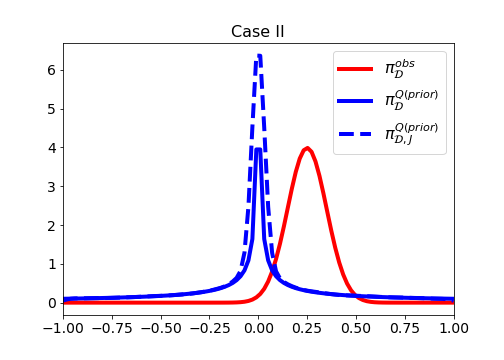}
\includegraphics[width=0.32\textwidth]{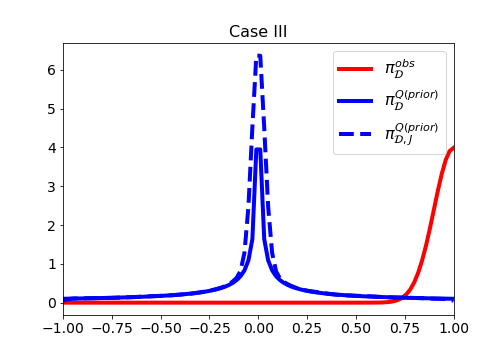}
\end{center}
\caption{Observed densities (red curves) used in Cases I (left), II (middle), and III (right). The solid blue curve is the exact push-forward for the exact map. The dashed blue curve is the estimated push-forward for this map using $m=1E5$ iid samples.
The errors in the estimated density for both this map and all approximate maps are primarily restricted to a narrow region around $q=0$, which is in the effective support of the observed density in Case II.
The effective support for the observed density in Case III extends beyond the range predicted by the exact push-forward density, which is a violation of Assumption~\ref{assump:dom}.}
\label{fig:concept_example_observed}
\end{figure}

In Case I (see the left plot in Figure~\ref{fig:concept_example_observed}),  Assumption~\ref{assump:dom} ``effectively'' holds (relative to the effective support of the observed density) and the estimated density is an accurate estimate of the associated push-forward density over the effective support of the observed density.
Case II (see the middle plot in Figure~\ref{fig:concept_example_observed}) is similar to Case I except that significant errors in the estimated density are now present in the effective support of the observed density.
In Case III (see the right plot in Figure~\ref{fig:concept_example_observed}), Assumption~\ref{assump:dom} is violated since significant portions of the effective support of the observed density are not ``predicted'' by any of the push-forward densities or their estimates.

\begin{table}[ht!]
\begin{center}
\begin{tabular}{l|c|c|c|c|c} \hline
\diagbox{Case}{$n$} & 1 & 2 & 4 & 8 & 16 \\ \hline
I  & 1.00   & 1.00   & 1.00    & 0.98  & 0.99  \\ \hline
II & 0.99   & 0.88   & 0.85    & 0.92  & 0.92   \\ \hline
III  & 0.69   & 0.66   & 0.63    & 0.61  & 0.61 \\ \hline
\end{tabular}
\end{center}
\caption{Representative estimates of $\mathbb{E}_i(r(\qoia))$ with $m=1E4$ iid samples used for both density estimation of associated push-forward densities and Monte Carlo estimate of $\mathbb{E}_i(r(\qoia))$.}
\label{tab:concept_example_mean_rs}
\end{table}
In Table~\ref{tab:concept_example_mean_rs}, we summarize the estimated values of $\text{I}(\postdensa)$ for $n=1, 2, 4, 8$, and $16$.
We see that in Case I, the values are all close to $1$, while in Case II, there is some deviation away from $1$ that is due to errors in the estimated push-forward densities present in the effective support of the observed density.
However, in Case III, we see significant deviations away from $1$ that are {\em primarily} due to the violation of Assumption~\ref{assump:dom}.
In fact, if we used the exact push-forward densities associated with each approximate map, then these values would be very close to $0.5$ because we are missing half the support of the observed density.
The reason for the values being slightly higher than $0.5$ is due to the kernel estimate ``extending'' the effective support of the push-forward densities beyond the actual data space defined by $[-1,1]$.

The takeaways are this: monitoring $\mathbb{E}_i(r(\qoia))$ is a good diagnostic for determining if either Assumption~\ref{assump:dom} is violated or if significant errors are present in the associated estimated push-forward density (and subsequently if errors are present in the associated estimate of the {\em approximate} updated density).
While there are certainly exceptions, we generally find large deviations of this expected value away from $1$ are due to violations of the assumption while smaller deviations away from $1$ are often due to numerical errors in the estimated density.
Finally, we note that just because estimated values of $\mathbb{E}_i(r(\qoia))$ are near 1 does {\em not} mean that the estimated push-forward density or associated estimate of the updated density are accurate approximations for the {\em exact} push-forward density or {\em exact} updated density one would obtain using the {\em exact} QoI map.
A ``good value'' of $\mathbb{E}_i(r(\qoia))$ (i.e., a value near $1$) simply speaks to both Assumption~\ref{assump:dom} not being violated and also to the relative accuracy of the estimated push-forward density with respect to its non-estimated counterpart.
\bibliographystyle{siam}
\bibliography{references}

\end{document}